\documentclass[12pt,thmsa]{article}
\usepackage{amssymb}
\usepackage{amsfonts}
\usepackage{amsmath}
\usepackage[top=3.8cm,bottom=3.8cm,textwidth=16cm,centering,a4paper]{geometry}
\usepackage{verbatim}
\usepackage[numbers,sort&compress]{natbib}
\usepackage{color}
\usepackage{indentfirst}
\usepackage{subfig}
\usepackage{tikz}

\allowdisplaybreaks[4]
\newtheorem{theorem}{Theorem}[section]

\newtheorem{proposition}[theorem]{Proposition}
\newtheorem{corollary}[theorem]{Corollary}
\newtheorem{remark}[theorem]{Remark}

\newtheorem{lemma}[theorem]{Lemma}

\newtheorem{definition}[theorem]{Definition}
\newenvironment{proof}[1][Proof]{\noindent\textbf{#1.} }{\ \rule{0.5em}{0.5em}}

\begin{document}

\title{Standing waves with prescribed mass for  biharmonic NLS with positive
dispersion and Sobolev critical exponent}
\date{}
\author{ Juntao Sun$^{a}$\thanks{%
E-mail address: jtsun@sdut.edu.cn(J. Sun)}, Shuai Yao$^{a}$\thanks{%
E-mail address: shyao@sdut.edu.cn (S. Yao)}, He Zhang$^{b}$\thanks{
E-mail address: hehzhang@163.com(H. Zhang)}  \\
{\footnotesize $^{a}$\emph{School of Mathematics and Statistics, Shandong
University of Technology, Zibo 255049, PR China }}\\
{\footnotesize $^b$\emph{School of Mathematics and Statistics, Central South
University, Changsha 410083, PR China }}}
\maketitle

\begin{abstract}
We investigate standing waves with prescribed mass for a class of biharmonic Schr\"{o}dinger
equations with positive Laplacian dispersion in the Sobolev critical regime. By establishing novel energy inequalities and developing a direct
 minimization approach, we prove the existence of two normalized solutions for the corresponding stationary problem. The first one is a ground state with negative
 level, and the second one is a higher-energy solution with positive level. It is worth noting that we do not work in the space of radial functions,
 and do not use Palais-Smale sequences so as to avoid applying the relatively complex mini-max approach based on
a strong topological argument. Finally, we explore the relationship between the ground states and the least action solutions,
 some asymptotic properties and dynamical behavior of solutions, such as the orbital stability and the global existence, are presented as well.
\end{abstract}



\section{Introduction}

Our starting point is the Cauchy problem for the biharmonic NLS (nonlinear
Schr\"{o}dinger equation) with mixed dispersion equations:%
\begin{equation}
\left\{
\begin{array}{ll}
i\partial _{t}\psi -\gamma \Delta ^{2}\psi -\mu \Delta \psi +|\psi
|^{p-2}\psi =0, & \forall (t,x)\in \mathbb{R}\times \mathbb{R}^{N}, \\
\psi (0,x)=\psi _{0}(x)\in H^{2}(\mathbb{R}^{N}), &
\end{array}%
\right.  \label{E1}
\end{equation}%
where $\gamma >0,\mu \in
\mathbb{R}
$ and $2<p<4^{\ast }:=\frac{2N}{(N-4)^{+}}.$ Here $4^{\ast }=\infty $ if $%
N\leq 4,$ and $4^{\ast }=\frac{2N}{N-4}$ if $N>5.$

The biharmonic NLS provides a canonical model for nonlinear Hamiltonian PDEs
with dispersion of super-quadratic order. The first study of biharmonic NLS
goes back to Karpman and Karpman-Shagalov \cite{K1996,K2000}, describing the
role of the fourth-order dispersion term in the propagation of a strong
laser beam in a bulk medium with Kerr nonlinearity. In recent years, a
considerable amount of work has been devoted to the study of problem (\ref%
{E1}). We refer the reader to \cite{B2000,P2007,Mi2011,R2016, S2010} for
global well-posedness and scattering, and to \cite{B2019,B2017,C2003} for
finite-time blow-up.

Taking into account the Hamiltonian structure of the biharmonic NLS, another
interesting topic on problem (\ref{E1}) is to study the standing waves of
the form $\psi (t,x)=e^{i\lambda t}u(x),$ where $\lambda \in \mathbb{R}$ is
a frequency and the function $u$ satisfies the following elliptic equation%
\begin{equation}
\begin{array}{ll}
\gamma \Delta ^{2}u+\mu \Delta u+\lambda u=|u|^{p-2}u, & \ \forall x\in
\mathbb{R}^{N}.%
\end{array}
\label{E2}
\end{equation}%
According to the role of the frequency $\lambda ,$ one can look for
solutions of equation (\ref{E2}) with $\lambda $ unknown. In this case $%
\lambda \in \mathbb{R}$ appears as a Lagrange multiplier and $L^{2}$-norms
of solutions are prescribed, which are usually called normalized solutions.
This study seems to be particularly meaningful from the physical point of
view, since solutions $\psi \in C([0,T);H^{2}(\mathbb{R}^{N}))$ of problem (%
\ref{E1}) conserve their mass along time, i.e. $\Vert \psi (t)\Vert
_{2}=\Vert \psi (0)\Vert _{2}$ for $t\in \lbrack 0,T)$. Moreover, this study
often offers a good insight of the dynamical properties of solutions for
problem (\ref{E1}), such as stability or instability.

Here we focus on this issue. For the sake of convenience, we take $\gamma=1$ since the coefficient $\gamma>0$
in equation (\ref{E2}) can be scaled out. Thus, for $a>0$ given, we consider the problem of
finding solutions to%
\begin{equation}
\left\{
\begin{array}{ll}
\Delta ^{2}u+\mu \Delta u+\lambda u=|u|^{p-2}u, & \ \forall x\in
\mathbb{R}^{N}, \\
\int_{\mathbb{R}^{N}}|u|^{2}=a^{2}. &
\end{array}%
\right.  \label{E3}
\end{equation}%
It is standard to show that solutions of problem (\ref%
{E3}) can be obtained as critical points of the energy functional $\Psi
_{\mu ,p}:H^{2}(\mathbb{R}^{N})\rightarrow \mathbb{R}$ given by
\begin{equation*}
\Psi _{\mu ,p}(u):=\frac{1}{2}\int_{\mathbb{R}^{N}}|\Delta u|^{2}dx-%
\frac{\mu }{2}\int_{\mathbb{R}^{N}}|\nabla u|^{2}dx-\frac{1}{p}\int_{\mathbb{%
R}^{N}}|u|^{p}dx
\end{equation*}%
on the constraint
\begin{equation*}
S_{a}:=\left\{ u\in H^{2}(\mathbb{R}^{N}):\int_{\mathbb{R}%
^{N}}|u|^{2}dx=a^{2}\right\} .
\end{equation*}

In order to better investigate the number and properties of solutions to
problem (\ref{E3}), we present the fibering map introduced in \cite%
{So2020-1,So2020,T1992} to help to understand the geometry of $\Psi _{\mu
,p}|_{S_{a}}$. Specifically, for each $u\in S_{a}$ and $s>0$, we set the
dilations%
\begin{equation*}
u_{s}(x)=s^{N/2}u(sx)\text{ for all }x\in \mathbb{R}^{N}.
\end{equation*}%
Define the fibering map $s\in (0,\infty )\rightarrow \phi _{u}(s):=\Psi
_{\mu ,p}(u_{s})$ given by%
\begin{equation*}
\phi _{u}(s):=\frac{s^{4}}{2}\Vert \Delta u\Vert _{2}^{2}-\frac{\mu s^{2}}{2}%
\Vert \nabla u\Vert _{2}^{2}-\frac{s^{2p\gamma _{p}}}{p}\Vert u\Vert
_{p}^{p},
\end{equation*}%
where $\gamma _{p}:=\frac{N(p-2)}{4p}.$ Moreover, we also define the
Pohozaev manifold%
\begin{equation*}
\mathcal{P}_{a}:=\{u\in S_{a}:Q_{p}(u)=0\}=\{u\in S_{a}:\phi _{u}^{\prime
}(1)=0\},
\end{equation*}%
where $Q_{p}(u)=2\Vert \Delta u\Vert _{2}^{2}-\mu \Vert \nabla u\Vert
_{2}^{2}-2\gamma _{p}\Vert u\Vert _{p}^{p}.$ It is well-known that any
critical point of $\Psi _{\mu ,p}|_{S_{a}}$ belongs to $\mathcal{P}_{a}$.
Notice that the critical points of $\phi _{u}$ allow to project a function
on $\mathcal{P}_{a}$. Then the monotonicity and the convexity properties of $%
\phi _{u}$ strongly affect the structure of $\mathcal{P}_{a}$ (and in turn
the geometry of $\Psi _{\mu ,p}|_{S_{a}}$), and also have a strong impact on
properties of the time-dependent problem (\ref{E1}). In addition, it is
natural to split $\mathcal{P}_{a}$ into three parts corresponding to local
minima, local maxima and points of inflection, namely%
\begin{align*}
\mathcal{P}_{a}^{+}& =\left\{ u\in \mathcal{P}_{a}:\phi _{u}^{\prime \prime
}(1)>0\right\} ; \\
\mathcal{P}_{a}^{0}& =\left\{ u\in \mathcal{P}_{a}:\phi _{u}^{\prime \prime
}(1)=0\right\} ; \\
\mathcal{P}_{a}^{-}& =\left\{ u\in \mathcal{P}_{a}:\phi _{u}^{\prime \prime
}(1)<0\right\} .
\end{align*}

\begin{figure}[h]
\centering
\subfloat{
\begin{minipage}[t]{0.3\textwidth}
		\centering
		\begin{tikzpicture}
			\draw[->] (-1,0) -- (3,0) node[right] {$s$};
			\draw[->] (0,-1.5) -- (0,1.5) node[above] {$\phi_{u}(s)$};
			\draw[domain=0:2.1, smooth, variable=\x, thick, red] plot ({\x},{\x^2 -1.5*\x});
			\node[below left] at (0,0) {$O$};
		\end{tikzpicture}
\caption*{(a)}
	\end{minipage}
}
\hfill
\subfloat{
\begin{minipage}[t]{0.3\textwidth}
		\centering
		\begin{tikzpicture}
			\draw[->] (-1,0) -- (3,0) node[right] {$s$};
			\draw[->] (0,-1.5) -- (0,1.5) node[above] {$\phi_{u}(s)$};
			\draw[domain=0:2.1, smooth, variable=\x, thick, red] plot ({\x},{-\x^2+1.5*\x});
			\node[below left] at (0,0) {$O$};
		\end{tikzpicture}
\caption*{(b)}
	\end{minipage}}
\hfill
\subfloat{
	\begin{minipage}[t]{0.3\textwidth}
		\centering
		\begin{tikzpicture}
			\draw[->] (-1,0) -- (3,0) node[right] {$s$};
			\draw[->] (0,-1.5) -- (0,1.5) node[above] {$\phi_{u}(s)$};
			\draw[domain=0:2.1, smooth, variable=\x, thick, red] plot ({\x},{-1*sin(200*\x )});
			\node[below left] at (0,0) {$O$};
			\node[right] at (2,1) { };
		\end{tikzpicture}
\caption*{(c)}
	\end{minipage}}
\caption{Possible forms of fibering map}
\end{figure}

In light of the sign of $\mu $ and the range of $p$, the properties of
fibering map $\phi _{u}$ can be summarized into the following three cases:

Case $(i):\mu \in \mathbb{R}$ and $2<p\leq \overline{p}:=2+\frac{8}{N}.$ The
fibering map $\phi _{u}$ has a unique critical point which is a global
minimum, see Figure 1(a). For this case, the energy functional $\Psi _{\mu
,p}$ is bounded from below on $S_{a}$. A ground state of problem (\ref{E3}%
) was found in \cite{B2018,Bo2019,Lu2023} for $\mu \leq 0$ and in \cite%
{F2022} for $\mu >0$ as the global minimizer of the minimization problem $%
m_{\mu ,p}(a)=\inf\limits_{u\in S_{a}}\Psi _{\mu ,p}(u)$ related to problem (%
\ref{E3}).

Case $(ii):\mu <0$ and $\overline{p}<p<4^{\ast }.$ The fibering map $\phi
_{u}$ has a unique critical point which is a global maximum, see Figure
1(b). For this case, the energy functional $\Psi _{\mu ,p}$ is bounded from
below on $\mathcal{P}_{a}$, but not on $S_{a}$, and $\mathcal{P}_{a}=%
\mathcal{P}_{a}^{-}$ is a natural constraint. In \cite{B2019}, a ground
state of problem (\ref{E3}) was found via using the Pohozaev manifold
method.

Case $(iii):\mu >0$ and $\overline{p}<p\leq \min \{4,4^{\ast }\}.$ The
fibering map $\phi _{u}$ has exactly two critical points: a local minimum at
$s_{1}=s_{1}(u)$ and a global maximum at $s_{2}=s_{2}(u)$. Moreover, $\phi
_{u}$ is decreasing in $(0,s_{1})\cup (s_{2},+\infty )$ and increasing in $%
(s_{1},s_{2}),$ see Figure 1(c). For this case, the energy functional $\Psi
_{\mu ,p}$ is bounded from below on $\mathcal{P}_{a}$, but not on $S_{a}$.
Furthermore, if the condition
\begin{equation}
\mu ^{p\gamma _{p}-2}a^{p-2}<C_{0}:=\left( \frac{p}{2\mathcal{C}%
_{N,p}^{p}(p\gamma _{p}-1)}\right) \left( \frac{p\gamma _{p}-2}{p\gamma
_{p}-1}\right) ^{p\gamma _{p}-2}  \label{E4}
\end{equation}%
holds, here $\mathcal{C}_{N,p}$ is as in (\ref{E6}) below, then $\Psi
_{\mu ,p}$ presents a convex-concave geometry and $\mathcal{P}_{a}=\mathcal{P%
}_{a}^{+}\cup \mathcal{P}_{a}^{-}$ is a natural constraint. Thus, one can
expect to find two critical points of $\Psi _{\mu ,p}|_{S_{a}}.$ Along this
direction, Luo and Yang \cite{LY2022} obtained two radial solutions of
problem (\ref{E3}) with $N\geq 5$. The first one is a local minimizer in $%
\mathcal{P}_{a}^{+}\cap H_{rad}^{2}(\mathbb{R}^{N})$, and the second one is
a mountain-pass type solution in $\mathcal{P}_{a}^{-}\cap H_{rad}^{2}(%
\mathbb{R}^{N})$, where $H_{rad}^{2}(\mathbb{R}^{N})=\{u\in H^{2}(\mathbb{R}%
^{N}):u$ is radially symmetric$\}.$ If there is no condition (\ref{E4}),
Fern\'{a}ndez et al. \cite{F2022} found a "best possible" local minimizer of
$\Psi _{\mu ,p}|_{S_{a}}$ by constructing an unbounded open set $\mathcal{%
O\subset }H^{2}(\mathbb{R}^{N})$ with $N\geq 5$ such that any possible local
minimizer of $\Psi _{\mu ,p}|_{S_{a}}$ must belong to $\mathcal{O}$. Very
recently, following the arguments in \cite{F2022}, Han et al. \cite{HGH2024}
proved that a local minimizer of $\Psi _{\mu ,p}|_{S_{a}}$ exists for all $%
N\geq 1,$ and the set of solutions is orbitally stable. Note that in all
papers \cite{F2022,HGH2024,LY2022}, only the Sobolev subcritical case $%
\overline{p}<p<\min \{4,4^{\ast }\}$ is studied. Thus the obtained results
in Case $(iii)$ leave a gap, say, the Sobolev critical case $p=4^{\ast }$
and $N\geq 9$. One of the main features of this paper is to fill this gap.

We also notice that in Case $(iii)$ there is no any results concerning on
ground states of problem (\ref{E3}) in $H^{2}(\mathbb{R}^{N}).$ In fact, a
radial ground state is found in the context of radial functions in \cite%
{LY2022}. The first aim of this paper is to study the existence of ground
states and the orbital stablility of the set of ground states for problem (%
\ref{E3}) with $\overline{p}<p\leq \min \{4,4^{\ast }\}.$ In addition, the
relationship between ground states and the least action solutions is
explored as well. Here the ground state and the least action solution are
defined in the following sense:

\begin{definition}
\label{D1.1}$(i)$ We say that $u\in S_{a}$ is a ground state to problem (\ref%
{E3}) if it is a solution having minimal energy among all the solutions
which belongs to $S_{a}$. Namely, if
\begin{equation*}
\Psi _{\mu ,p}(u)=\inf \{\Psi _{\mu ,p}(u):u\in S_{a}\text{ and }\left( \Psi
_{\mu ,p}|_{S_{a}}\right) ^{\prime }(u)=0\}.
\end{equation*}%
$(ii)$ Given $\lambda \in \mathbb{R}$, a nontrivial solution $u\in H^{2}(%
\mathbb{R}^{N})$ to equation (\ref{E2}) is called an least action solution
if it achieves the infimum of the action functional
\begin{equation*}
I_{\lambda }(u):=\Psi _{\mu ,p}(u)+\frac{\lambda }{2}\int_{\mathbb{R}%
^{N}}|u|^{2}dx
\end{equation*}%
among all the nontrivial solutions, that is
\begin{equation*}
I_{\lambda }(u)=\theta _{\lambda }:=\inf \{I_{\lambda }(v):v\in H^{2}(%
\mathbb{R}^{N})\backslash \{0\}\text{ and }I_{\lambda }^{\prime }(v)=0\}.
\end{equation*}
\end{definition}

In addition to ground states, another aim of this paper is to find a
high-energy solution of problem (\ref{E3}) with $\overline{p}<p\leq \min
\{4,4^{\ast }\}.$ In \cite{LY2022}, a mountain-pass type radial solution is
found in $H_{rad}^{2}(\mathbb{R}^{N})$ by means of the Palais-Smale
sequences. Distinguishing from \cite{LY2022}, by introducing some new ideas,
we do not work in $H_{rad}^{2}(\mathbb{R}^{N}),$ but in $H^{2}(\mathbb{R}%
^{N}),$ and do not need to consider Palais-Smale sequences, so that we avoid
applying the relatively complex mini-max approach based on a strong
topological argument as in \cite{BS2017,CJ2019,Gh1993}. Finally, with the
help of the high-energy soution, we establish the conditions of global
existence for the Cauchy problem (\ref{E1}).

\subsection{Main results}

First of all, we focus on the existence of ground states for problem (\ref{E3}).

\begin{theorem}
\label{T1.2} Let $\overline{p}<p<4$ for $5\leq N\leq 8$, or $\overline{p}%
<p\leq 4^{\ast }$ for $N\geq 9.$ Assume that condition (\ref{E4}) holds.
Then the following statements are true.\newline
$(i)$ $\Psi _{\mu ,p}$ restricted to $S_{a}$ has a ground state $u_{\mu
}^{+} $. This ground sate is a local minimizer of $\Psi _{\mu ,p}$ in the
set $\mathcal{M}_{a}:=\{u\in S_{a}:\Vert \Delta u\Vert _{2}<\rho _{0}\}$ for
some $\rho _{0}>0,$ namely%
\begin{equation*}
\Psi _{\mu ,p}(u_{\mu }^{+})=\alpha _{p}(a):=\inf\limits_{u\in \mathcal{M}%
_{a}}\Psi _{\mu ,p}(u)=m_{\mu ,p}(a):=\inf\limits_{u\in \mathcal{P}%
_{a}^{+}}\Psi _{\mu ,p}(u).
\end{equation*}%
Moreover, any ground state for $\Psi _{\mu ,p}$ on $S_{a}$ is a local
minimizer of $\Psi _{\mu ,p}$ on $\mathcal{M}_{a}.$\newline
$(ii)$ The ground state $u_{\mu }^{+}$ is a real-valued sign-changing
solution to problem (\ref{E3}) for some $\lambda =\lambda _{\mu }^{+}>\frac{\mu
^{2}}{4}.$\newline
$(iii)$ $\alpha _{p}(a)\rightarrow 0^{-},$ $\lambda _{\mu }^{+}\rightarrow
0^{+}$ and $\Vert \Delta u_{\mu }^{+}\Vert _{2}\rightarrow 0$ as $\mu
\rightarrow 0^{+}$.
\end{theorem}

\begin{remark}
\label{T1.3}It is worth noting that we are not sure whether the ground
state $u_{\mu }^{-}$ provided by Theorem \ref{T1.2} is or not radially
symmetric, since the biharmonic operator $\Delta ^{2}$ is involved, which is
an open question. In fact, for the Sobolev subcritical case $\overline{p}%
<p<\min \{4,4^{\ast }\},$ if $\frac{p}{2}\in \mathbb{N}$, then such solution
is radially symmetric via the Fourier rearrangement technique \cite{Bu2022}.
However, under the assumption on $p$ in Theorem \ref{T1.2}, we find $\frac{p%
}{2}\not\in \mathbb{N}$, and thus the Fourier rearrangement technique can
not be applied.
\end{remark}

In the proof of Theorem \ref{T1.2}, we use the minimizing method and the
truncation techniques \cite{J2022}. Unlike \cite{J2022}, since problem (\ref%
{E3}) involves a positive second-order dispersion term, it is more
difficult to rule out the vanishing property of the minimizing sequence. To
overcome this difficulty, we establish a biharmonic variant of Lion's lemma
and the upper estimate of $\alpha _{p}(a)$.

We now turn to find a high-energy solution of problem (\ref{E3}). Set
\begin{equation*}
M_{\mu ,p}(a):=\inf\limits_{u\in \mathcal{P}_{a}^{-}}\Psi _{\mu ,p}(u).
\end{equation*}%
Then we have the following result.

\begin{theorem}
\label{T1.4} Assume that $\overline{p}<p<4^{\ast }$ for $N\geq 2,$ or $%
p=4^{\ast }$ for $N\geq 9$. Then there exists $\mu _{a}\in \left(
0,(C_{0}a^{2-p})^{\frac{1}{p\gamma _{p}-2}}\right) $ such that for any $\mu
\in (0,\mu _{a}),$ the following statements hold.\newline
$(i)$ $\Psi _{\mu ,p}|_{S_{a}}$ has a critical point $u_{\mu }^{-}$ at
positive level $M_{\mu ,p}(a)>0$.\newline
$(ii)$ $u_{\mu }^{-}$ is a real-valued sign-changing solution to problem (\ref%
{E3}) for some $\lambda =\lambda _{\mu }^{-}>0$.\newline
$(iii)$ If $\overline{p}<p<4^{\ast }$, then $M_{\mu ,p}(a)\rightarrow
M_{p}(a)$ and $u_{\mu }^{-}\rightarrow u^{-}$ as $\mu \rightarrow 0^{+}$,
where $M_{p}(a)=\Psi _{\mu ,p}(u^{-})$ with $u^{-}$ being a ground state to
problem (\ref{E3}) with $\mu =0.$\newline
$(iv)$ If $p=4^{\ast }$, then $M_{\mu ,p}(a)\rightarrow \frac{2}{N}\mathcal{S%
}^{\frac{N}{4}}$, $\lambda _{\mu }^{-}\rightarrow 0$ and $u_{\mu
}^{-}\rightarrow u_{\epsilon }$ as $\mu \rightarrow 0^{+}$, where $\mathcal{S%
}$ is as in (\ref{E7}) and $u_{\epsilon }$ defined as (\ref{E8}), is the
unique solution to the equation $\Delta ^{2}u=|u|^{4^{\ast }-2}u$ in $%
\mathbb{R}^{N}$.
\end{theorem}

In the proof of Theorem \ref{T1.4}, we use the direct minimization of the
energy functional $\Psi _{\mu ,p}$ on the set $\mathcal{P}_{a}^{-}.$ The
main difficulty lies on how to recover the compactness of the minimizing
sequence, since we work in the space $H^{2}(\mathbb{R}^{N})$ and do not use
the Palais-Smale sequences related to the Pohozaev identity. To solve this
problem, a novel strategy is to establish the following energy inequalities
\begin{equation*}
M_{\mu ,p}(a)<M_{\mu ,p}(a_{1})+m_{\mu ,p}(a_{2})\ \text{for}\ \text{any }%
a_{1},a_{2}>0\ \text{and}\ a_{1}^{2}+a_{2}^{2}=a^{2},
\end{equation*}%
and
\begin{equation*}
M_{\mu ,4^{\ast }}(a)<m_{\mu ,4^{\ast }}(a)+\frac{2}{N}\mathcal{S}^{\frac{N}{%
4}}.
\end{equation*}%
As far as we know, it is the first time that such energy inequalities are
given.

Next, we establish the relationship between the ground states and the least
action solutions.

\begin{theorem}
\label{T1.5} Let $\lambda (u)$ be the Lagrange multiplier corresponding to
an arbitrary minimizer $u\in S_{a}$ of problem (\ref{E3}). Under the
assumptions of Theorem \ref{T1.2}, the following statements hold.\newline
$(i)$ For given $\lambda \in \{\lambda (u):u\in S_{a}\ \text{is a minimizer
of problem}\ $(\ref{E3})$\}$, there exists $\mu _{a}^{\ast }>0$ satisfying
\begin{equation*}
\mu _{a}^{\ast }\left\{
\begin{array}{ll}
=\left( C_{0}a^{2-p}\right) ^{\frac{1}{p\gamma _{p}-2}} & \text{ if }p\in
(\overline{p},4^{\ast }), \\
\in \left( 0,\left( C_{0}a^{2-p}\right) ^{\frac{1}{p\gamma _{p}-2}%
}\right) & \text{ if }p=4^{\ast }.%
\end{array}%
\right.
\end{equation*}%
such that for any $\mu \in (0,\mu _{a}^{\ast })$, any least action solution $%
u_{\lambda }\in H^{2}(\mathbb{R}^{N})$ of equation (\ref{E2}) is a $\text{%
minimizer of problem}\ $(\ref{E3}), that is
\begin{equation*}
\Vert u_{\lambda \Vert _{2}}=a\text{ and}\ \Psi _{\mu ,p}(u_{\lambda
})=\alpha _{p}(a).
\end{equation*}%
$(ii)$ For any $\mu \in \left( 0,(C_{0}a^{2-p})^{\frac{1}{p\gamma _{p}-2}%
}\right) ,$ any minimizer $u_{a}\in S_{a}$ of problem (\ref{E3}) is a
least action solution of equation (\ref{E2}) with $\lambda =\lambda
(u_{a})>0$.
\end{theorem}

Finally, we study the dynamical behavior of the solutions to problem (\ref%
{E1}). As we know, the local well-posedness of the Cauchy problem has been
proved in \cite{P2007} for $\overline{p}<p\leq 4^*$.

\begin{definition}
\label{D1.3}$Z\subset (\mathbb{R}^{N})$ is stable if: $Z\neq 0$ and for any $%
v\in Z$ and any $\epsilon >0$, there exists a $\kappa >0$ such that if $\psi
\in H^{2}(\mathbb{R}^{N})$ satisfies $\Vert \psi -v\Vert _{H^{2}}<\kappa $,
then $\zeta _{\psi }(t)$ is globally defined and $\inf\limits_{z\in Z}\Vert
\zeta _{\psi }(t)-z\Vert _{H^{2}}<\epsilon $ for all $t\in \mathbb{R}^{N}$,
where $\zeta _{\psi }(t)$ is the solution to problem (\ref{E1})
corresponding to the initial condition $\psi $.
\end{definition}

\begin{theorem}
\label{T1.7} Under the assumptions of Theorem \ref{T1.2}, the set%
\begin{equation*}
\mathbf{B}_{\rho_{0}}:=\{u\in \mathcal{M}_{a}:\Psi _{\mu ,p}(u)=\alpha
_{p}(a)\}
\end{equation*}%
is compact, up to translation, and it is orbitally stable.
\end{theorem}

Define%
\begin{equation*}
\Gamma _{a}:=\left\{ u\in S_{a}:\Psi _{\mu ,p}(u)<M_{\mu ,p}(a)\ \text{and}\
Q_{p}(u)>0\right\} .
\end{equation*}%
Then we consider global existence.

\begin{theorem}
\label{T1.8} Under the assumptions of Theorem \ref{T1.4}, the solution $%
\psi \in C\left( [0,T];H^{2}(\mathbb{R}^{N})\right) $ to problem (\ref{E1}%
) with the initial datum $\psi _{0}\in \Gamma _{a}$ exists globally in time.
\end{theorem}

The paper is organized as follows. After introducing some preliminary
results in Section 2, we prove Theorem \ref{T1.2} in section 3 and Theorem %
\ref{T1.4} in section 4, respectively. We prove Theorem \ref{T1.5} in
Section 5. Finally in Section 6, we study dynamical behavior and prove
Theorems \ref{T1.7} and \ref{T1.8}.

\section{Preliminary results}

In this section, we give some preliminary results.

$(i)$ The interpolation inequality (\cite{Bo2019}):%
\begin{equation}
\Vert \nabla u\Vert _{2}^{2}\leq \Vert u\Vert _{2}\Vert \Delta u\Vert _{2},\
\forall u\in H^{2}(\mathbb{R}^{N}).  \label{E5}
\end{equation}

$(ii)$ The Gagliardo-Nirenberg inequality (\cite[Theorem in Lecture II]{N1}):%
\begin{equation}
\Vert u\Vert _{p}^{p}\leq \mathcal{C}_{N,p}^{p}\Vert u\Vert _{2}^{p(1-\gamma
_{p})}\Vert \Delta u\Vert _{2}^{p\gamma _{p}},\ \forall u\in H^{2}(\mathbb{R}%
^{N}),  \label{E6}
\end{equation}%
where $2<p\leq 4^{\ast }$ and $\mathcal{C}_{N,p}$ is a constant depending on
$N$ and $p$. In particular, when $p=4^{\ast }$, it holds $\mathcal{C}%
_{N,4^{\ast }}=\mathcal{S}^{-1/2}$ and $\gamma _{4^{\ast }}=1$, where
\begin{equation}
\mathcal{S}:=\inf\limits_{u\in H^{2}(\mathbb{R}^{N})\setminus \{0\}}\frac{%
\int_{\mathbb{R}^{N}}|\Delta u|^{2}dx}{\left( \int_{\mathbb{R}%
^{N}}|u|^{4^{\ast }}dx\right) ^{2/4^{\ast }}}.  \label{E7}
\end{equation}%
Let
\begin{equation}
u_{\epsilon }(x):=(N(N-4)(N^{2}-4))^{\frac{N-4}{8}}\epsilon ^{\frac{N-4}{4}%
}\left( \frac{1}{\epsilon +|x|^{2}}\right) ^{\frac{N-4}{2}}.  \label{E8}
\end{equation}%
It is well known that $u_{\epsilon }$ is a solution of
\begin{equation*}
\begin{array}{ll}
\Delta ^{2}u=|u|^{4^{\ast }-2}u\  & \text{in}\ \mathbb{R}^{N}%
\end{array}%
\end{equation*}%
satisfying
\begin{equation*}
\Vert \Delta u_{\epsilon }\Vert _{2}^{2}=\Vert u_{\epsilon }\Vert _{4^{\ast
}}^{4^{\ast }}=\mathcal{S}^{\frac{N}{4}}.
\end{equation*}

Fixed $u\in S_{a},$ for any $s>0$, since $\mu >0$ and $\overline{p}<p\leq
\min \{4,4^{\ast }\},$ we have
\begin{eqnarray*}
\Psi _{\mu ,p}(u_{s}(x)) &=&\frac{s^{4}}{2}\Vert \Delta u\Vert _{2}^{2}-%
\frac{\mu s^{2}}{2}\Vert \nabla u\Vert _{2}^{2}-\frac{s^{2p\gamma _{p}}}{p}%
\Vert u\Vert _{p}^{p} \\
&\rightarrow &-\infty \text{ as }s\rightarrow \infty .
\end{eqnarray*}%
which implies that $\Psi_{\mu,p}$ is unbounded below on $S_{a}.$ So, we need
to consider the local minimization problems.

For any $u\in S_{a}$, by using the interpolation and Gagliardo-Nirenberg
inequalities (\ref{E5})--(\ref{E6}), we have
\begin{align*}
\Psi _{\mu ,p}(u)& \geq \Vert \Delta u\Vert _{2}\left( \frac{1}{2}\Vert
\Delta u\Vert _{2}-\frac{\mu a}{2}-\frac{\mathcal{C}_{N,p}^{p}}{p}%
a^{p(1-\gamma _{p})}\Vert \Delta u\Vert _{2}^{p\gamma _{p}-1}\right) \\
& =\Vert \Delta u\Vert _{2}h_{a}(\Vert \Delta u\Vert _{2}),
\end{align*}%
where the function%
\begin{equation*}
h_{a}(s):=\frac{1}{2}s-\frac{\mu a}{2}-\frac{\mathcal{C}_{N,p}^{p}a^{p(1-%
\gamma _{p})}s^{p\gamma _{p}-1}}{p}\ \text{ for }s>0.
\end{equation*}

\begin{lemma}
\label{T2.1} Let $\overline{p}<p\leq 4^{\ast }$ and $a,\mu >0$. Then the
function $h_{a}$ has a unique global maximum point, and the maximum value
satisfies
\begin{equation}
\left\{
\begin{array}{ll}
\max\limits_{s>0}h_{a}(s)>0, & \ \text{if}\ a<a_{0}, \\
\max\limits_{s>0}h_{a}(s)=0, & \ \text{if}\ a=a_{0}, \\
\max\limits_{s>0}h_{a}(s)<0, & \ \text{if}\ a>a_{0},%
\end{array}%
\right.  \label{E9}
\end{equation}%
where
\begin{equation*}
a_{0}:=\left( \frac{p\gamma _{p}-2}{\mu (p\gamma _{p}-1)}\right) ^{\frac{%
p\gamma _{p}-2}{p-2}}\left( \frac{p}{2\mathcal{C}_{N,p}^{p}(p\gamma _{p}-1)}%
\right) ^{\frac{1}{p-2}}.  \label{e2.2}
\end{equation*}
\end{lemma}

\begin{proof}
Clearly, it holds $p\gamma _{p}-1>1$ since $\overline{p}<p\leq 4^{\ast }$.
Then we obtain that $h_{a}(s)\rightarrow 0$ as $s\rightarrow 0$ and $%
h_{a}(s)\rightarrow -\infty $ as $s\rightarrow +\infty $. The derivative of $%
h_{a}(s)$ is given by
\begin{equation*}
h_{a}^{\prime }(s)=\frac{1}{2}-\frac{\mathcal{C}_{N,p}^{p}(p\gamma
_{p}-1)a^{p(1-\gamma _{p})}{s}^{p\gamma _{p}-2}}{p},
\end{equation*}%
which indicates that $h_{a}^{\prime }(s)=0$ has a unique solution
\begin{equation}
\rho _{a}=\left( \frac{pa^{p(\gamma _{p}-1)}}{2\mathcal{C}_{N,p}^{p}(p\gamma
_{p}-1)}\right) ^{\frac{1}{p\gamma _{p}-2}}.  \label{E10}
\end{equation}%
Moreover, $h_{a}^{\prime }(s)>0$ for $s\in (0,\rho _{a})$ and $h_{a}^{\prime
}(s)<0$ for $s\in (\rho _{a},\infty )$. Hence, we conclude that $\rho _{a}$
is the unique global maximum point of $h_{a}(s)$ and the maximum value is
\begin{align}
\max\limits_{s>0}h_{a}(s)=h_{a}(\rho _{a})=& \frac{1}{2}\rho _{a}-\frac{\mu a%
}{2}-\frac{\mathcal{C}_{N,p}^{p}a^{p(1-\gamma _{p})}\rho _{a}^{p\gamma
_{p}-1}}{p}  \notag \\
=& \frac{p\gamma _{p}-2}{2(p\gamma _{p}-1)}\rho _{a}-\frac{\mu a}{2}.
\label{E11}
\end{align}%
By (\ref{E11}), and together with the definition of $a_{0}$, we easily
obtain (\ref{E9}). The proof is complete.
\end{proof}

By Lemma \ref{T2.1}, we set $\rho _{0}:=\rho _{a_{0}}>0$. According to the
definition of $a_{0}$, we know that $a<a_{0}$ is equivalent to condition (%
\ref{E4}). Since $\overline{p}<p\leq 4^{\ast }$, there holds $0\leq
p(1-\gamma _{p})<1$. This shows that $h_{a}(\rho _{0})$ is strictly
decreasing for $a\in (0,a_{0})$. Thus, combining with (\ref{E9}), one has
\begin{equation}
h_{a}(\rho _{0})>h_{a_{0}}(\rho _{0})=0\text{ for all}\ 0<a<a_{0}.
\label{E12}
\end{equation}%
So for each $a\in (0,a_{0})$, we set%
\begin{equation*}
\mathcal{M}_{a}:=\{u\in S_{a}:\Vert \Delta u\Vert _{2}<\rho _{0}\}\text{ and
}\alpha _{p}(a):=\inf\limits_{u\in \mathcal{M}_{a}}\Psi _{\mu ,p}(u).
\end{equation*}

\begin{lemma}
\label{T2.2} Assume that $\overline{p}<p\leq 4^{\ast }$ and condition (\ref%
{E4}) holds. Then the following statements are true.\newline
$(i)$
\begin{equation}
\alpha _{p}(a):=\inf\limits_{u\in \mathcal{M}_{a}}\Psi _{\mu
,p}(u)<0<\inf\limits_{u\in \partial \mathcal{M}_{a}}\Psi _{\mu ,p}(u).
\label{E13}
\end{equation}
$(ii)$ If $\alpha _{p}(a)$ is attained, then any ground states of
problem (\ref{E3}) are all contained in $\mathcal{M}_{a}$.
\end{lemma}

\begin{proof}
$(i)$ Since $a\in (0,a_{0})$, for any $u\in \partial \mathcal{M}_{a}$, it
follows from (\ref{E12}) that
\begin{align*}
\Psi _{\mu ,p}(u)& \geq \Vert \Delta u\Vert _{2}h_{a}(\Vert \Delta u\Vert
_{2}) \\
& =\rho _{0}h_{a}(\rho _{0}) \\
& >0.
\end{align*}%
Now we claim that $\alpha _{p}(a)<0$. For any $u\in S_{a}$ and $s>0$, we
notice that
\begin{equation*}
\Psi _{\mu ,p}(u_{s})=\frac{s^{4}}{2}\Vert \Delta u\Vert _{2}^{2}-\frac{\mu
s^{2}}{2}\Vert \nabla u\Vert _{2}^{2}-\frac{s^{2p\gamma _{p}}}{p}\Vert
u\Vert _{p}^{p}.
\end{equation*}%
Then there exists $s_{\rho _{0}}>0$ such that $\Vert \Delta u_{s_{\rho
_{0}}}\Vert _{2}<\rho_{0}$ and $\Psi _{\mu ,p}(u_{s_{\rho _{0}}})<0$.
Hence, we conclude that $\alpha _{p}(a)<0$.

$(ii)$ It is well known that all critical points of $\Psi _{\mu ,p}$
restricted to $S_{a}$ belong to $\mathcal{P}_{a}$. Then we calculate that
for any $u\in \mathcal{P}_{a}$, there holds
\begin{align}
\Psi _{\mu ,p}(u)& =\left( \frac{1}{2}-\frac{1}{p\gamma _{p}}\right) \Vert
\Delta u\Vert _{2}^{2}-\frac{\mu }{2}\left( 1-\frac{1}{p\gamma _{p}}\right)
\Vert \nabla u\Vert _{2}^{2}  \notag \\
& \geq \left( \frac{1}{2}-\frac{1}{p\gamma _{p}}\right) \Vert \Delta u\Vert
_{2}^{2}-\frac{a\mu }{2}\left( 1-\frac{1}{p\gamma _{p}}\right) \Vert \Delta
u\Vert _{2}.  \label{E14}
\end{align}%
According to (\ref{E14}), we deduce that if $\Psi _{\mu ,p}(u)<0$, it holds
$\Vert \Delta u\Vert _{2}<\frac{(p\gamma _{p}-1)\mu a}{p\gamma _{p}-2}$.
Since $a<a_{0}$, we have $\frac{(p\gamma _{p}-1)\mu a}{p\gamma _{p}-2}<\rho
_{0}$. This shows that $\Vert \Delta u\Vert _{2}<\rho _{0}$ and thus $u\in
\mathcal{M}_{a}$. The proof is complete.
\end{proof}

\begin{remark}
\label{T2.3} From Lemma \ref{T2.2} $(ii)$, we easily obtain that
\begin{equation*}
\inf\limits_{u\in \mathcal{P}_{a}}\Psi _{\mu ,p}(u)=\inf\limits_{u\in
\mathcal{M}_{a}}\Psi _{\mu ,p}(u)=\alpha _{p}(a).
\end{equation*}
\end{remark}

\begin{lemma}
\label{T2.4} (\cite[Lemma 2.3]{LY2022}) Let $\overline{p}<p<4^{\ast }$ for $%
N\geq 2$, or $p=4^{\ast }$ for $N\geq 5$. Assume that condition (\ref{E4})
holds. Then $\mathcal{P}_{a}^{0}=\emptyset $, and $\mathcal{P}_{a}$ is a
smooth manifold of codimension $2$ in $H^{2}(\mathbb{R}^{N}).$
\end{lemma}

\begin{lemma}
\label{T2.5} (\cite[Lemma 2.4]{LY2022}) Let $\overline{p}<p<4^{\ast }$ for $%
N\geq 2$, or $p=4^{\ast }$ for $N\geq 5.$ Assume that condition (\ref{E4})
holds. If $u\in \mathcal{P}_{a}$ is a critical point for $\Psi _{\mu ,p}|_{%
\mathcal{P}_{a}}$, then $u$ is a critical point for $\Psi _{\mu
,p}|_{S_{a}}. $
\end{lemma}

We now investigate the fiber map $\phi _{u}$ and determine the location and
type of the point $\Psi _{\mu ,p}|_{s_{a}}$. Using the interpolation and
Gagliardo-Nirenberg inequalities (\ref{E5})--(\ref{E6}), for each $u\in
S_{a}$, there holds
\begin{align*}
\Psi _{\mu ,p}(u)& \geq \frac{1}{2}\Vert \Delta u\Vert _{2}^{2}-\frac{\mu a}{%
2}\Vert \Delta u\Vert _{2}^{2}-\frac{\mathcal{C}_{N,p}^{p}}{p}a^{p(1-\gamma
_{p})}\Vert \Delta u\Vert _{2}^{p\gamma _{p}} \\
& =\frac{1}{2}s^{2}-\frac{\mu a}{2}s-\frac{\mathcal{C}_{N,p}^{p}}{p}%
a^{p(1-\gamma _{p})}s^{p\gamma _{p}} \\
& =:g_{p}(s).
\end{align*}%
Then, we firstly consider the properties of the function $g_{p}:\mathbb{R}%
^{+}\rightarrow \mathbb{R}$.

\begin{lemma}
\label{T2.6} Let $\overline{p}<p<4^{\ast }$ for $N\geq 2$, or $p=4^{\ast }$
for $N\geq 5$. Assume that condition (\ref{E4}) holds. Then the function $%
g_{p}$ has a local strict minimum point $s_{1}$ and a global strict maximum
point $s_{2}$. Moreover, the following statements are true.\newline
$(i)$ There exist two positive constants $R_{0}$ and $R_{1}$ such that
\begin{equation*}
0<s_{1}<R_{0}<\overline{s}<s_{2}<R_{1},\text{ }g_{p}(R_{0})=0=g_{p}(R_{1}),%
\text{ }g_{p}(s_{1})<g_{p}(R_{0})<g_{p}(\overline{s})<g_{p}(s_{2}),
\end{equation*}%
and $g_{p}(s)<0$ for each $s\in (0,R_{0})\cup (R_{1},\infty )$, $g_{p}(s)>0$
for each $s\in (R_{0},R_{1})$. Here $\overline{s}:=\left[ \frac{pa^{p(\gamma
_{p}-1)}}{2(p\gamma _{p}-1)\mathcal{C}_{N,p}^{p}}\right] ^{\frac{1}{p\gamma
_{p}-2}}$.\newline
$(ii)$ $\lim\limits_{\mu \rightarrow 0^{+}}g_{p}(s_{1})=0$ and $%
\lim\limits_{\mu \rightarrow 0^{+}}g_{p}(s_{2})=\frac{p\gamma _{p}-2}{%
2p\gamma _{p}}\left( \frac{a^{p(\gamma _{p}-1)}}{\mathcal{C}_{N,p}^{p}\gamma
_{p}}\right) ^{\frac{2}{p\gamma _{p}-2}}>0$.
\end{lemma}

\begin{proof}
$(i)$ The case of $\overline{p}<p<4^{\ast }$ have been shown in \cite{LY2022}%
. Moreover, the proof of the case $p=4^{\ast }$ is similar to that of the
case of $\overline{p}<p<4^{\ast }$. So, we omit it here.

$(ii)$ Since $\lim\limits_{\mu \rightarrow 0}g_{p}(s)=\frac{1}{2}s^{2}-\frac{%
\mathcal{C}_{N,p}^{p}}{p}a^{p(1-\gamma _{p})}s^{p\gamma _{p}}$, we easily
obtain the conclusion of $(ii)$.
\end{proof}

\begin{lemma}
\label{T2.7} Let $\overline{p}<p<4^{\ast }$ for $N\geq 2$, or $p=4^{\ast }$
for $N\geq 5$. Assume that condition (\ref{E4}) holds. For every $u\in
S_{a}$, the function $\phi _{u}(s)$ has exactly two critical points $0<s_{u}<%
\overline{s}_{u}$. Moreover, the following statements are true.\newline
$(i)$ $\phi _{u}^{\prime }(s_{u})=\phi _{u}^{\prime }(\overline{s}_{u})=0$, $%
\phi _{u}^{\prime }(s)<0$ for each $s\in (0,s_{u})\cup (\overline{s}%
_{u},+\infty )$ and $\phi _{u}^{\prime }(s)>0$ for each $s\in (s_{u},%
\overline{s}_{u})$.\newline
$(ii)$ $u_{s_{u}}\in \mathcal{P}_{a}^{+}$ and $u_{\overline{s}_{u}}\in
\mathcal{P}_{a}^{-}$. And if $u_{s}\in \mathcal{P}_{a}$, then either $%
s=s_{u} $ or $s=\overline{s}_{u}$.\newline
$(iii)$%
\begin{equation*}
\Vert \Delta u_{s_{u}}\Vert _{2}<\frac{\mu a(p\gamma _{p}-1)}{p\gamma _{p}-2}%
<\rho_{a}<\Vert \Delta u_{\overline{s}_{u}}\Vert _{2},
\end{equation*}%
where $\rho_{a}$ is given in (\ref{E10}).\newline
$(iv)$%
\begin{equation*}
g_{p}(s_{1})\leq \Psi_{\mu,p}(u_{s_{u}})=\min \left\{ \Psi _{\mu
}(u_{s}):s\geq 0\ \text{and}\ \Vert \Delta u_{s}\Vert _{2}<\rho_{a}\right\} <0
\end{equation*}%
and
\begin{equation*}
\Psi _{\mu ,p}(u_{\overline{s}_{u}})=\max \left\{ \Psi _{\mu
,p}(u_{s}):s\geq 0\right\} =\max \left\{ \Psi _{\mu ,p}(u_{s}):s\geq 0\
\text{and}\ \Vert \Delta u_{s}\Vert _{2}>\rho_{a}\right\} \geq
g_{p}(s_{2})>0.
\end{equation*}%
Here, $g_{p}(s_{1})$ and $g_{p}(s_{2})$ are given in Lemma \ref{T2.6}.%
\newline
$(v)$ The map $u\in S_{a}\mapsto s_{u}$ and $u\in S_{a}\mapsto \overline{s}%
_{u}$ are of class $C^{1}$.
\end{lemma}

\begin{proof}
Let $u\in S_{a}$, then there holds that $u_{s}\in \mathcal{P%
}_{a}^{+}$ if and only if $\phi _{u}^{\prime }(s)=0$. Clearly, it is easy to
obtain that $\lim\limits_{s\rightarrow 0}\phi _{u}(s)=0$ and $%
\lim\limits_{s\rightarrow \infty }\phi _{u}(s)=-\infty $. Thus, we only need
to prove that there exists $s>0$ such that $\phi _{u}^{\prime }(s)>0$. Set $%
\tilde{f}(s)=2s^{2}\Vert \Delta u\Vert _{2}^{2}-2\gamma _{p}s^{2p\gamma
_{p}-2}\Vert u\Vert _{p}^{p}-\mu \Vert \nabla u\Vert _{2}^{2}$ for $s>0$.
Clearly, $\lim\limits_{s\rightarrow 0}\tilde{f}(s)<0$ and $%
\lim\limits_{s\rightarrow \infty }\tilde{f}(s)=-\infty $. Moreover, since (\ref{E4}) holds, we have
\begin{align*}
\tilde{f}(\tilde{s}_{u})=\max\limits_{s\geq 0}\tilde{f}(s)=& \left( \frac{%
2p\gamma _{p}-4}{p\gamma _{p}-1}\right) \left( \frac{\Vert \Delta u\Vert
_{2}^{2}}{\gamma _{p}(p\gamma _{p}-1)\Vert u\Vert _{p}^{p}}\right) ^{\frac{1%
}{p\gamma _{p}-2}}\Vert \Delta u\Vert _{2}^{2}-\mu \Vert \nabla u\Vert
_{2}^{2} \\
\geq & \left[ \frac{2p\gamma _{p}-4}{p\gamma _{p}-1}\left( \frac{1}{\gamma
_{p}(p\gamma _{p}-1)\mathcal{C}_{N,p}^{p}a^{p(1-\gamma _{p})}}\right) ^{%
\frac{1}{p\gamma _{p}-2}}-\mu a\right] \Vert \Delta u\Vert _{2} \\
>& 0,
\end{align*}%
where $\tilde{s}_{u}:=\left( \frac{\Vert \Delta u\Vert _{2}^{2}}{\gamma
_{p}(p\gamma _{p}-1)\Vert u\Vert _{p}^{p}}\right) ^{\frac{1}{2(p\gamma
_{p}-2)}}$. Then there exist $s_{u}<\tilde{s}_{u}<\overline{s}_{u}$ such
that $\tilde{f}(s_{u})=\tilde{f}(\overline{s}_{u})=0$. Moreover, there hold $%
\tilde{f}(s)<0$ for each $s\in (0,s_{u})\cup (\overline{s}_{u},+\infty )$
and $\tilde{f}(s)>0$ for each $s\in (s_{u},\overline{s}_{u})$. Since $\phi
_{u}^{\prime }(s)=s\tilde{f}(s)$, we have $\phi _{u}^{\prime }(s_{u})=\phi
_{u}^{\prime }(\overline{s}_{u})=0$, $\phi _{u}^{\prime }(s)<0$ for each $%
s\in (0,s_{u})\cup (\overline{s}_{u},+\infty )$ and $\phi _{u}^{\prime
}(s)>0 $ for each $s\in (s_{u},\overline{s}_{u})$. This implies that $%
u_{s_{u}}\in \mathcal{P}_{a}^{+}$, $u_{\overline{s}_{u}}\in \mathcal{P}%
_{a}^{-}$ and $\phi _{u}(s)$ has only two critical points. So, $(i)$ and $%
(ii)$ hold.

Since $\lim\limits_{s\rightarrow 0^{-}}\Psi _{\mu ,p}(u_{s})=0^{-}$ and $%
\phi _{u}^{\prime }(s)<0$ for each $s\in (0,s_{u})$, we have $\Psi _{\mu
,p}(u_{s_{u}})<0$. Combining with $u_{s_{u}}\in \mathcal{P}_{a}^{+}$, we
have
\begin{align*}
0& >\frac{p\gamma _{p}-2}{2p\gamma _{p}}\Vert \Delta u_{s_{u}}\Vert _{2}^{2}-%
\frac{p\gamma _{p}-1}{2p\gamma _{p}}\mu \Vert \nabla u_{s_{u}}\Vert _{2}^{2}
\\
& \geq \frac{p\gamma _{p}-2}{2p\gamma _{p}}\Vert \Delta u_{s_{u}}\Vert
_{2}^{2}-\frac{p\gamma _{p}-1}{2p\gamma _{p}}\mu a\Vert \Delta u_{s_{u}}\Vert_{2},
\end{align*}%
which implies that $\Vert \Delta u_{s_{u}}\Vert _{2}<\frac{\mu a(p\gamma
_{p}-1)}{p\gamma _{p}-2}$. A direct calculation shows that
\begin{align*}
\Psi _{\mu ,p}(u_{s})& =\frac{1}{2}\Vert \Delta u_{s}\Vert _{2}^{2}-\frac{1}{%
2}\mu \Vert \nabla u_{s}\Vert _{2}^{2}-\frac{1}{p}\Vert u_{s}\Vert _{p}^{p}
\\
& \geq \frac{1}{2}\Vert \Delta u_{s}\Vert _{2}^{2}-\frac{1}{2}\mu a\Vert
\Delta u_{s}\Vert _{2}-\frac{\mathcal{C}_{N,p}^{p}a^{p(1-\gamma _{p})}}{p}%
\Vert \Delta u_{s}\Vert _{2}^{p\gamma _{p}} \\
& =\frac{1}{2}s^{4}\Vert \Delta u\Vert _{2}^{2}-\frac{1}{2}\mu as^{2}\Vert
\Delta u\Vert _{2}-\frac{\mathcal{C}_{N,p}^{p}a^{p(1-\gamma _{p})}}{p}%
s^{2p\gamma _{p}}\Vert \Delta u\Vert _{2}^{p\gamma _{p}}.
\end{align*}%
Let $\tilde{s}=s^{2}\Vert \Delta u\Vert _{2}$, we get $\Psi _{\mu
,p}(u_{s})\geq g_{p}(\tilde{s})$. By Lemma \ref{T2.6}, one has $0>\Psi _{\mu
,p}(u_{s_{u}})\geq g_{p}(s_{1})$ and $\Psi _{\mu ,p}(u_{\overline{s}%
_{u}})\geq g_{p}(s_{2})>0$. Since $u_{\overline{s}_{u}}\in \mathcal{P}%
_{a}^{-}$, one has
\begin{align*}
0& <-\frac{1}{2}\Vert \Delta u_{\overline{s}_{u}}\Vert _{2}^{2}+\frac{%
p\gamma _{p}-1}{p}\Vert u_{\overline{s}_{u}}\Vert _{p}^{p} \\
& \leq -\frac{1}{2}\Vert \Delta u_{\overline{s}_{u}}\Vert _{2}^{2}+\frac{%
p\gamma _{p}-1}{p}\mathcal{C}_{N,p}^{p}a^{p(1-\gamma _{p})}\Vert \Delta u_{%
\overline{s}_{u}}\Vert _{2}^{p\gamma _{p}},
\end{align*}%
which implies that $\Vert \Delta u_{\overline{s}_{u}}\Vert _{2}>\left( \frac{%
p}{2(p\gamma _{p}-1)\mathcal{C}_{N,p}^{p}a^{p(1-\gamma _{p})}}\right) ^{%
\frac{1}{p\gamma _{p}-2}}$. Since $\mu ^{p\gamma _{p}-2}a^{p-2}<C_{0}$, we
deduce that
\begin{equation*}
\Vert \Delta u_{\overline{s}_{u}}\Vert _{2}<\frac{\mu a(p\gamma _{p}-1)}{%
p\gamma _{p}-2}<\left( \frac{p}{2(p\gamma _{p}-1)\mathcal{C}%
_{N,p}^{p}a^{p(1-\gamma _{p})}}\right) ^{\frac{1}{p\gamma _{p}-2}}=\rho_{a}<\Vert \Delta u_{\overline{s}_{u}}\Vert _{2}.
\end{equation*}%
So, $(iii)$ holds. From $(ii)$ and $(iii)$, it is easily seen that $(iv)$
holds. Finally, we can apply the implicit function theorem on the $C^{1}$
function $\Phi (s,u):=\phi _{u}^{\prime }(s)$ to show that the map $u\in
S_{a}\mapsto s_{u}$ and $u\in S_{a}\mapsto \overline{s}_{u}$ are both of
class $C^{1}$. Similar to \cite[Lemma 2.6 (4)]{LY2022}, $(v)$ holds. The
proof is complete.
\end{proof}

\begin{corollary}
\label{L3.5} Let $\overline{p}<p<4^{\ast }$ for $N\geq 2$, or $p=4^{\ast }$
for $N\geq 5$. Assume that condition (\ref{E4}) holds. Then the following
statements are true.\newline
$(i)$%
\begin{equation*}
\mathcal{P}_{a}^{+}\subset \mathcal{M}_{a}:=\left\{ u\in S_{a}:\Vert \Delta
u\Vert _{2}<\rho_{0}\right\} ;
\end{equation*}%
\newline
$(ii)$%
\begin{equation*}
0>m_{\mu ,p}(a)=\inf\limits_{u\in \mathcal{P}_{a}}\Psi _{\mu
,p}(u)=\inf\limits_{u\in \mathcal{P}_{a}^{+}}\Psi _{\mu
,p}(u)=\inf\limits_{u\in \mathcal{M}_{a}}\Psi _{\mu ,p}(u)=\alpha
_{p}(a)\geq g_{p}(s_{1});
\end{equation*}%
\newline
$(iii)$%
\begin{equation*}
M_{\mu ,p}(a)\geq g_{p}(s_{2})>0.
\end{equation*}%
Here, $g_{p}(s_{1}),g_{p}(s_{2})$ are given in Lemma \ref{T2.6}.
\end{corollary}

\section{The existence of ground states}

\begin{lemma}
\label{T3.1} Assume that $\overline{p}<p\leq 4^{\ast }$ and condition (\ref%
{E4}) hold. Then the following statements are true.\newline
$(i)$ The mapping $\alpha _{p}(a)$ is continuous in $(0,a_{0});$\newline
$(ii)$%
\begin{equation}
\alpha _{p}(a)\leq \beta ^{2}\alpha _{p}(a/\beta )<\alpha _{p}(a/\beta
),\quad \forall \beta >1,  \label{E15}
\end{equation}%
and thus
\begin{equation}
\alpha _{p}(a)\leq \alpha _{p}(\beta _{1}a)+\alpha _{p}(\beta _{2}a),\quad
\forall \beta _{1},\beta _{2}>0\ \text{and}\ \beta _{1}^{2}+\beta _{2}^{2}=1.
\label{E16}
\end{equation}%
Moreover, if $\alpha _{p}(\beta _{1}a)$ or $\alpha _{p}(\beta _{2}a)$ is
reached, then the inequality is strict.
\end{lemma}

\begin{proof}
$(i)$ Similar to the argument of \cite[Lemma 2.6]{J2022}, we easily conclude
that $\alpha _{p}(a)$ is a continuous mapping in $(0,a_{0})$.

$(ii)$ It follows from (\ref{E13}) that $\alpha _{p}(a)<0$ for any $a\in
(0,a_{0})$. Thus, to prove (\ref{E15}), we assume that for any $\epsilon >0$
sufficiently small, there exists $u\in \mathcal{M}_{\beta ^{-1}a}$ such that
$\Psi _{\mu ,p}(u)\leq \alpha _{p}(a/\beta )+\epsilon <0$. This shows that
\begin{equation}
h_{\beta ^{-1}a}(\Vert \Delta u\Vert _{2})<0\ \text{and}\ \Vert \Delta
u\Vert _{2}<\rho _{0}.  \label{E17}
\end{equation}%
By (\ref{E12}), we calculate that for $\beta ^{-1}\leq k\leq 1$,
\begin{align*}
h_{\beta ^{-1}a}(k\rho _{0})& =\frac{1}{2}\rho _{0}-\beta ^{-1}k\frac{\mu a}{%
2}-\beta ^{-p(1-\gamma _{p})}k^{p\gamma _{p}-2}\frac{\mathcal{C}%
_{N,p}^{p}a^{p(1-\gamma _{p})}{\rho _{0}}^{p\gamma _{p}-1}}{p} \\
& \geq \frac{1}{2}\rho _{0}-\frac{\mu a}{2}-\frac{\mathcal{C}%
_{N,p}^{p}a^{p(1-\gamma _{p})}{\rho _{0}}^{p\gamma _{p}-1}}{p} \\
& =h_{a}(\rho _{0})>0,
\end{align*}%
which implies $h_{\beta ^{-1}a}(s)>0$ for any $s\in \lbrack \beta ^{-1}\rho
_{0},\rho _{0}]$. Then in view of (\ref{E17}), we have
\begin{equation}
\Vert \Delta u\Vert _{2}<\frac{\rho _{0}}{\beta }.  \label{E18}
\end{equation}%
Let $v=\beta u$. Then combining with (\ref{E18}), one has
\begin{equation*}
\Vert v\Vert _{2}=\beta \Vert u\Vert _{2}=\beta a\ \text{and}\ \Vert \Delta
v\Vert _{2}=\beta \Vert \Delta u\Vert _{2}<\rho _{0}.
\end{equation*}%
This shows $v\in \mathcal{M}_{a}$ and thus
\begin{align*}
\alpha _{p}(a)\leq \Psi _{\mu ,p}(v)=& \frac{\beta ^{2}}{2}\Vert \Delta
u\Vert _{2}^{2}-\frac{\mu \beta ^{2}}{2}\Vert \nabla u\Vert _{2}^{2}-\frac{%
\beta ^{p}}{p}\Vert u\Vert _{p}^{p} \\
<& \beta ^{2}\Psi _{\mu ,p}(u)\leq \beta ^{2}\alpha _{p}(a/\beta )+\beta
^{2}\epsilon .
\end{align*}%
By the arbitrariness of $\epsilon $, we conclude that $\alpha _{p}(a)\leq
\beta ^{2}\alpha _{p}(a/\beta )$. Moreover, if $\alpha _{p}(a/\beta )$ is
attained, then it holds $\alpha _{p}(a)<\beta ^{2}\alpha _{p}(a/\beta )$.
Without loss of generality, let $0<\beta _{1}<\beta _{2}$ and $\beta
_{1}^{2}+\beta _{2}^{2}=1$. Then we have
\begin{equation*}
\alpha _{p}(a)\leq \frac{\beta _{1}^{2}+\beta _{2}^{2}}{\beta _{1}^{2}}%
\alpha _{p}(\beta _{1}a)=\alpha _{p}(\beta _{1}a)+\frac{\beta _{2}^{2}}{%
\beta _{1}^{2}}\alpha _{p}\left( \frac{\beta _{1}}{\beta _{2}}\beta
_{2}a\right) \leq \alpha _{p}(\beta _{1}a)+\alpha _{p}(\beta _{2}a),
\end{equation*}%
which implies that (\ref{E16}) holds. The proof is complete.
\end{proof}

\begin{lemma}
\label{T3.2} Let $\overline{p}<p<4$ for $5\leq N\leq 8$, or $\overline{p}%
<p\leq 4^{\ast }$ for $N\geq 9.$ Assume that condition (\ref{E4}) holds.
Then we have $\alpha _{p}(a)<-\frac{\mu ^{2}a^{2}}{8}$.
\end{lemma}

\begin{proof}
The proof is similar to that of \cite[Lemma 3.2]{LY2022}. So we omit it here.
\end{proof}

Now we establish a biharmonic variant of Lion's lemma.

\begin{lemma}
\label{T3.3} Let $2\leq p\leq 4^{\ast }$. Assume that $\{u_{n}\}$ is bounded
in $H^{2}(\mathbb{R}^{N})$ and for some $r>0$,
\begin{equation}
\lim\limits_{n\rightarrow \infty }\sup\limits_{y\in \mathbb{R}%
^{N}}\int_{B_{r}(y)}|u_{n}|^{p}dx=0.  \label{E19}
\end{equation}%
Then it holds
\begin{equation*}
\int_{\mathbb{R}^{N}}|u_{n}|^{q}dx=o_{n}(1),
\end{equation*}%
where
\begin{equation*}
q\in
\begin{cases}
(2,4^{\ast }),\ \text{if}\ 2\leq p<4^{\ast }, \\
(2,4^{\ast }],\ \text{if}\ p=4^{\ast }.%
\end{cases}%
\end{equation*}
\end{lemma}

\begin{proof}
The case of $2\leq p<4^{\ast }$ has been proved in \cite{Me2023}. So, we
only consider the case of $p=4^{\ast }$. By calculating, we have
\begin{align}
\int_{B_{r}(y)}|u_{n}|^{4^{\ast }}dx& =\left(
\int_{B_{r}(y)}|u_{n}|^{4^{\ast }}dx\right) ^{\frac{4}{N}}\left(
\int_{B_{r}(y)}|u_{n}|^{4^{\ast }}dx\right) ^{\frac{N-4}{N}}  \notag \\
& \leq C\left( \int_{B_{r}(y)}|u_{n}|^{4^{\ast }}dx\right) ^{\frac{4}{N}%
}\int_{B_{r}(y)}|\Delta u_{n}|^{2}dx.  \label{E20}
\end{align}%
Cover $\mathbb{R}^{N}$ with balls of radius $r$ so that each point of $%
\mathbb{R}^{N}$ contains at most $N+1$ balls. Using (\ref{E19}) and (\ref%
{E20}), we obtain
\begin{equation*}
\int_{\mathbb{R}^{N}}|u_{n}|^{4^{\ast }}dx\leq C(N+1)\sup\limits_{y\in
\mathbb{R}^{N}}\left( \int_{B_{r}(y)}|u_{n}|^{4^{\ast }}dx\right) ^{\frac{4}{%
N}}\int_{\mathbb{R}^{N}}|\Delta u_{n}|^{2}dx=o_{n}(1),
\end{equation*}%
which implies that
\begin{equation}
\int_{\mathbb{R}^{N}}|u_{n}|^{4^{\ast }}dx=o_{n}(1).  \label{E21}
\end{equation}%
By (\ref{E21}) and the following interposition inequality
\begin{equation*}
\int_{\mathbb{R}^{N}}|u|^{q}dx\leq \left( \int_{\mathbb{R}%
^{N}}|u|^{2}dx\right) ^{\frac{2N-(N-4)q}{8}}\left( \int_{\mathbb{R}%
^{N}}|u|^{4^{\ast }}dx\right) ^{\frac{8-2N+(N-4)q}{8}},\ \forall 2<q<4^{\ast
},
\end{equation*}%
we easily deduce that
\begin{equation}
\int_{\mathbb{R}^{N}}|u_{n}|^{q}dx=o_{n}(1)\ \text{for}\ 2<q<4^{\ast }.
\label{E22}
\end{equation}%
From (\ref{E21}) and (\ref{E22}), we complete the proof.
\end{proof}

\begin{proposition}
\label{T3.4} Let $\overline{p}<p<4$ for $5\leq N\leq 8$, or $\overline{p}%
<p\leq 4^{\ast }$ for $N\geq 9.$ Assume that condition (\ref{E4}) holds.
Let $\{u_{n}\}\subset \mathcal{M}_{a}$ satisfy $\Psi _{\mu
,p}(u_{n})\rightarrow \alpha _{p}(a)$ as $n\rightarrow \infty $. Then up to
subsequence, there exists $u\in \mathcal{M}_{a}$ such that $u_{n}\rightarrow
u$ in $H^{2}(\mathbb{R}^{N})$ and $\Psi _{\mu ,p}(u)=\alpha _{p}(a)$.
\end{proposition}

\begin{proof}
Clearly, $\{u_{n}\}$ is bounded in $H^{2}(\mathbb{R}^{N}),$ since $%
\{u_{n}\}\subset \mathcal{M}_{a}$. We first claim that there exist a $d>0$
and a sequence $\{x_{n}\}\subset \mathbb{R}^{N}$ such that
\begin{equation*}
\int_{B_{r}(x_{n})}|u_{n}|^{p}dx\geq d>0\ \text{for some}\ r>0.
\end{equation*}%
Otherwise, (\ref{E19}) holds. Then we have $u_{n}\rightarrow 0$ in $L^{p}(%
\mathbb{R}^{N})$. This indicates that
\begin{align*}
\Psi _{\mu ,p}(u_{n})& =\frac{1}{2}\Vert \Delta u_{n}\Vert _{2}^{2}-\frac{%
\mu }{2}\Vert \nabla u_{n}\Vert _{2}^{2}+o_{n}(1) \\
& \geq \frac{1}{2}\Vert \Delta u_{n}\Vert _{2}^{2}-\frac{\mu a}{2}\Vert
\Delta u_{n}\Vert _{2}+o_{n}(1) \\
& \geq -\frac{\mu ^{2}a^{2}}{8},
\end{align*}%
which contradicts with $\alpha _{p}(a)<-\frac{\mu ^{2}a^{2}}{8}$ by Lemma %
\ref{T3.2}. Thus there exist $d>0$ and $\{x_{n}\}\subset \mathbb{R}^{N}$
such that
\begin{equation*}
u_{n}(x-x_{n})\rightharpoonup u\neq 0.
\end{equation*}%
Let $w_{n}=u_{n}(x-x_{n})-u$. It follows from Brezis-Lieb lemma that
\begin{align}
\Vert u_{n}\Vert _{2}^{2}& =\Vert u\Vert _{2}^{2}+\Vert w_{n}\Vert
_{2}^{2}+o_{n}(1),  \label{E23} \\
\Vert \Delta u_{n}\Vert _{2}^{2}& =\Vert \Delta u\Vert _{2}^{2}+\Vert \Delta
w_{n}\Vert _{2}^{2}+o_{n}(1),  \label{E24} \\
\Psi _{\mu ,p}(u_{n})& =\Psi _{\mu ,p}(u_{n}(x-x_{n}))=\Psi _{\mu
,p}(w_{n})+\Psi _{\mu ,p}(u)+o_{n}(1).  \label{E25}
\end{align}%
Next, we claim that
\begin{equation*}
\Vert w_{n}\Vert _{2}^{2}=o_{n}(1).
\end{equation*}%
In order to prove this, let us denote $a_{1}^{2}:=\Vert u\Vert _{2}^{2}>0$.
By (\ref{E23}), we easily get $\Vert w_{n}\Vert
_{2}^{2}=a^{2}-a_{1}^{2}+o_{n}(1)$. If $a_{1}=a$, then the claim follows. We
assume by contradiction that $a_{1}<a$. According to $\Vert \Delta
u_{n}\Vert _{2}<\rho _{0}$ and (\ref{E24}), one has $\Vert \Delta
w_{n}\Vert _{2},\Vert \Delta u\Vert _{2}<\rho _{0}$. Thus, we conclude that
\begin{equation}
\Psi _{\mu ,p}(w_{n})\geq \alpha _{p}(\Vert w_{n}\Vert _{2})\text{ and }\Psi
_{\mu ,p}(u)\geq \alpha _{p}(a_{1}).  \label{E26}
\end{equation}%
Since $\alpha _{p}(a)$ is continuous, according to (\ref{E25}) and (\ref%
{E26}), we obtain that
\begin{align}
\alpha _{p}(a)& :=\Psi _{\mu ,p}(w_{n})+\Psi _{\mu ,p}(u)+o_{n}(1)\geq
\alpha _{p}\left( \sqrt{a^{2}-a_{1}^{2}}\right) +\Psi _{\mu ,p}(u)+o_{n}(1)
\notag \\
& \geq \alpha _{p}\left( \sqrt{a^{2}-a_{1}^{2}}\right) +\alpha
_{p}(a_{1})+o_{n}(1).  \label{E27}
\end{align}%
Using (\ref{E16}), we deduce that
\begin{equation}
\alpha _{p}\left( \sqrt{a^{2}-a_{1}^{2}}\right) +\alpha _{p}(a_{1})\geq
\alpha _{p}(a).  \label{E28}
\end{equation}%
If $\Psi _{\mu ,p}(u)>\alpha _{p}(a_{1})$, then it follows from (\ref{E27}%
) and (\ref{E28}) that $\alpha _{p}(a)<\alpha _{p}(a),$ which is
impossible. If $\Psi _{\mu ,p}(u)=\alpha _{p}(a_{1})$, then $\alpha
_{p}(a_{1})$ is attained and thus
\begin{equation}
\alpha _{p}\left( \sqrt{a^{2}-a_{1}^{2}}\right) +\alpha _{p}(a_{1})>\alpha
_{p}(a).  \label{E29}
\end{equation}%
So by (\ref{E27}) and (\ref{E29}), we reach a contradiction. Hence we
have $\Vert u\Vert _{2}^{2}=a^{2}$ and $\Vert w_{n}\Vert _{2}^{2}=o_{n}(1)$.
By the interpolation inequality and Sobolev inequality, we have
\begin{equation}
\Vert \nabla w_{n}\Vert _{2}^{2}=o_{n}(1)  \label{E30}
\end{equation}%
and
\begin{equation}
\Vert w_{n}\Vert _{q}^{q}=o_{n}(1)\ \text{for}\ 2\leq q<4^{\ast }.
\label{E31}
\end{equation}%
Now, we show that $\Vert \Delta w_{n}\Vert _{2}^{2}=o_{n}(1)$. If $p<4^{\ast
}$, then it follows from (\ref{E24}), (\ref{E30})-(\ref{E31}) and the
interpolation inequality that
\begin{align*}
\alpha _{p}(a)& =\Psi _{\mu ,p}(u_{n})+o_{n}(1) \\
& =\Psi _{\mu ,p}(u)+\frac{1}{2}\Vert \Delta w_{n}\Vert _{2}^{2}+o_{n}(1) \\
& \geq \alpha _{p}(a)+\frac{1}{2}\Vert \Delta w_{n}\Vert _{2}^{2}+o_{n}(1),
\end{align*}%
which implies that $\Vert \Delta w_{n}\Vert _{2}^{2}=o_{n}(1)$. If $%
p=4^{\ast }$, then it follows from (\ref{E24}) and (\ref{E30}) that
\begin{align}
\alpha _{4^{\ast }}(a)& =\Psi _{\mu ,p}(u_{n})+o_{n}(1)  \notag \\
& =\Psi _{\mu ,p}(u)+\frac{1}{2}\Vert \Delta w_{n}\Vert _{2}^{2}-\frac{1}{%
4^{\ast }}\Vert w_{n}\Vert _{4^{\ast }}^{4^{\ast }}+o_{n}(1)  \notag \\
& \geq \alpha _{4^{\ast }}(a)+\frac{1}{2}\Vert \Delta w_{n}\Vert _{2}^{2}-%
\frac{1}{4^{\ast }}\Vert w_{n}\Vert _{4^{\ast }}^{4^{\ast }}+o_{n}(1)  \notag
\\
& \geq \alpha _{4^{\ast }}(a)+\frac{1}{2}\Vert \Delta w_{n}\Vert _{2}^{2}-%
\frac{1}{4^{\ast }}\mathcal{S}^{-4^{\ast }/2}\Vert \Delta w_{n}\Vert
_{2}^{4^{\ast }}.  \label{E32}
\end{align}%
Note that $\rho _{0}=\left( \frac{N}{N+4}\right) ^{(N-4)/8}\mathcal{S}^{N/8}$
when $p=4^{\ast }$. This shows that $\Vert \Delta w_{n}\Vert <\rho _{0}<%
\mathcal{S}^{N/8}$. Let $f(s)=\frac{1}{2}s^{2}-\frac{1}{4^{\ast }}\mathcal{S}%
^{-4^{\ast }/2}s^{4^{\ast }}$ for $s>0$. Then we have $f(\mathcal{S}%
^{N/8})=\max_{s>0}f(s)$ and $f(s)>0$ for any $0<s<\mathcal{S}^{N/8}$. In
view of (\ref{E32}), it holds $\Vert \Delta w_{n}\Vert _{2}^{2}=o_{n}(1)$.
This indicates that $u_{n}\rightarrow u$ in $H^{2}(\mathbb{R}^{N})$ and $%
u\in \mathcal{M}_{a}$. The proof is complete.
\end{proof}

According to Lemma \ref{T3.4}, we have the following corollary.

\begin{corollary}
\label{T3.5} Let $\overline{p}<p<4^{\ast }$ for $N\geq 5$, or $p=4^{\ast }$
for $N\geq 9$. Assume that condition (\ref{E4}) holds. Then the set
\begin{equation*}
\mathbf{B}_{\rho_{0}}:=\{u\in \mathcal{M}_{a}:\Psi _{\mu ,p}(u)=\alpha
_{p}(a)\}
\end{equation*}%
is compact in $H^{2}(\mathbb{R}^{N})$, up to translation.
\end{corollary}

\textbf{We are ready to prove Theorem \ref{T1.2}}: $(i)$ According to Lemma %
\ref{T2.2} and Proposition \ref{T3.4}, we obtain that $\Psi _{\mu
,p}|_{s_{a}}$ has a ground state solution $u_{\mu }^{+}\in \mathcal{M}_{a}$.
We now claim that the corresponding Lagrange multipliers $\lambda _{\mu
}^{+}>\frac{\mu ^{2}}{4}$. Since $(\lambda _{\mu }^{+},u_{\mu }^{+})$ is the
solution of problem (\ref{E3}), we have
\begin{equation}
\lambda _{\mu }^{+}a^{2}=-\Vert \Delta u_{\mu }^{+}\Vert _{2}^{2}+\mu \Vert
\nabla u_{\mu }^{+}\Vert _{2}^{2}+\Vert u_{\mu }^{+}\Vert _{p}^{p}
\label{E33}
\end{equation}%
and
\begin{equation}
0=Q_{p}(u_{\mu }^{+})=2\Vert \Delta u_{\mu }^{+}\Vert _{2}^{2}-\mu \Vert \nabla
u_{\mu }^{+}\Vert _{2}^{2}-2\gamma _{p}\Vert u_{\mu }^{+}\Vert _{p}^{p}.
\label{E34}
\end{equation}%
It follows from (\ref{E33}) and (\ref{E34}) that%
\begin{equation}
\lambda _{\mu }^{+}a^{2}=\frac{\mu }{2}\Vert \nabla u_{\mu }^{+}\Vert
_{2}^{2}+(1-\gamma _{a})\Vert u_{\mu }^{+}\Vert _{p}^{p}\text{ and }\Psi
_{\mu ,p}(u_{\mu }^{+})=-\frac{\mu }{4}\Vert \nabla u_{\mu }^{+}\Vert
_{2}^{2}+\frac{p\gamma _{p}-2}{2p}\Vert u_{\mu }^{+}\Vert _{p}^{p}.
\label{E35}
\end{equation}%
According to Lemma \ref{T3.2} and (\ref{E35}), we deduce that
\begin{equation}
\lambda _{\mu }^{+}a^{2}\geq -2\Psi _{\mu ,p}(u_{\mu }^{+})>\frac{\mu
^{2}a^{2}}{4},  \label{E36}
\end{equation}%
which shows that $\lambda _{\mu }^{+}>\frac{\mu ^{2}}{4}$.

$(ii)$ By employing the argument of \cite[Theorem 3.7]{B2018}, we easily
obtain that $u_{\mu }^{+}$ is sign-changing.

$(iii)$ For any $u_{\mu }^{+}\in \mathbf{B}_{\rho_{0}}$, we have $u_{\mu
}^{+}\in \mathcal{P}_{a}$. Then we calculate that
\begin{align*}
0>\Psi _{\mu ,p}(u_{\mu }^{+})=& \frac{p\gamma _{p}-2}{2p\gamma _{p}}\Vert
\Delta u_{\mu }^{+}\Vert _{2}^{2}-\frac{\mu (p\gamma _{p}-1)}{2p\gamma _{p}}%
\Vert \nabla u_{\mu }^{+}\Vert _{2}^{2} \\
\geq & \frac{p\gamma _{p}-2}{2p\gamma _{p}}\Vert \Delta u_{\mu }^{+}\Vert
_{2}^{2}-\frac{\mu a(p\gamma _{p}-1)}{2p\gamma _{p}}\Vert \Delta u_{\mu
}^{+}\Vert _{2},
\end{align*}%
which implies that $\Vert \Delta u_{\mu }^{+}\Vert _{2}<\frac{(p\gamma
_{p}-1)\mu a}{p\gamma _{p}-2}$. This shows that%
\begin{equation}
\lim\limits_{\mu \rightarrow 0^{+}}\Vert \Delta u_{\mu }^{+}\Vert _{2}=0,
\label{E37}
\end{equation}%
and thus
\begin{equation}
\lim\limits_{\mu \rightarrow 0^{+}}\Vert \nabla u_{\mu }^{+}\Vert
_{2}=\lim\limits_{\mu \rightarrow 0^{+}}\Vert u_{\mu }^{+}\Vert _{p}=0.
\label{E38}
\end{equation}%
Since
\begin{equation*}
\lambda _{\mu }^{+}a^{2}=-\Vert \Delta u_{\mu }^{+}\Vert _{2}^{2}+\mu \Vert
\nabla u_{\mu }^{+}\Vert _{2}^{2}+\Vert u_{\mu }^{+}\Vert _{p}^{p},
\end{equation*}%
by (\ref{E36})-(\ref{E38}), we have $\lambda _{\mu }^{+}\rightarrow
0^{+}$ as $\mu \rightarrow 0^{+}.$ According to $(i)$, we know that $\alpha
_{p}(a)<-\frac{\mu ^{2}a^{2}}{8}$. So it follows that $\alpha
_{p}(a)\rightarrow 0^{-}$ as $\mu \rightarrow 0^{+}.$ The proof is complete.

\section{The existence of high-energy solution}

\begin{lemma}
\label{T4.1} Let $\overline{p}<p<4^{\ast }$ for $N\geq 2$, or $p=4^{\ast }$
for $N\geq 5$. Assume that condition (\ref{E4}) holds. Then the following
statements are true.\newline
$(i)$ The mapping $M_{\mu ,p}(a):a\mapsto M_{\mu ,p}(a)$ is continuous.%
\newline
$(ii)$ There holds%
\begin{equation*}
M_{\mu ,p}(a)\leq \beta ^{2+4b}M_{\mu ,p}(\beta ^{-1}a),\ \forall {\beta >1},
\end{equation*}%
where $b:=\frac{2(p-2)}{8-N(p-2)}.$ In particular, the inequality is strict
if $M_{\mu ,p}(\beta ^{-1}a)$ is attained. Moreover, there hold
\begin{equation*}
M_{\mu ,p}(a)<M_{\mu ,p}(\beta ^{-1}a),\ \forall \beta {>1}
\end{equation*}%
for $\overline{p}<p<4^{\ast }$ and
\begin{equation*}
M_{\mu ,4^{\ast }}(a)\leq M_{\mu ,4^{\ast }}(\beta ^{-1}a),\ \forall \beta {%
>1},
\end{equation*}%
where the inequality is strict if $M_{\mu ,4^{\ast }}(\beta ^{-1}a)$ is
attained.
\end{lemma}

\begin{proof}
$(i)$ Let $a\in \left( 0,\left( C_{0}\mu ^{-(p\gamma _{p}-2)}\right)^{\frac{1}{p-2}}\right) $ be arbitrary and $a_{n}\in \left( 0,\left( C_{0}\mu ^{-(p\gamma _{p}-2)}\right)^{\frac{1}{p-2}}\right) $ be such that $a_{n}\rightarrow a$
as $n\rightarrow \infty $. From the definition of $M_{\mu ,p}(a_{n})>0$, for
any $\epsilon >0$ sufficiently small, there exists $u_{n}\in \mathcal{P}%
_{a_{n}}^{-}$ such that
\begin{equation*}
\Psi _{\mu ,p}(u_{n})\leq M_{\mu ,p}(a_{n})+\epsilon \ \text{and }\Psi_{\mu ,p}(u_{n})>0.
\end{equation*}%
Set $v_{n}=\left( \frac{a}{a_{n}}\right) ^{c}u_{n}\left( \left( \frac{a}{%
a_{n}}\right) ^{b}\cdot \right) ,$ where $c:=\frac{8}{8-N(p-2)}.$ Then there
exists $s_{n}>0$ such that $\left( v_{n}\right) _{s_{n}}\in \mathcal{P}%
_{a}^{+}$. Since $a_{n}\rightarrow a$ and $Q_{p}(u_{n})=0$, we have $\Psi _{\mu
,p}(v_{n})=\Psi _{\mu ,p}(u_{n})+o_{n}(1)$ and $%
Q_{p}(v_{n})=Q_{p}(u_{n})+o_{n}(1)=o_{n}(1)$. Thus, it follows from Lemma \ref{T2.7}
$(iv)$ that $s_{n}=1+o_{n}(1)$. Since $Q_{p}(u_{n})=0$, we have
\begin{align*}
\Psi _{\mu ,p}(u_{n})& =\frac{p\gamma _{p}-2}{2p\gamma _{p}}\Vert \Delta
u_{n}\Vert _{2}^{2}-\frac{\mu (p\gamma _{p}-1)}{2p{\gamma _{p}}}\Vert \nabla
u_{n}\Vert _{2}^{2} \\
& \geq \frac{p\gamma _{p}-2}{2p\gamma _{p}}\Vert \Delta u_{n}\Vert _{2}^{2}-%
\frac{\mu a(p\gamma _{p}-1)}{2p{\gamma _{p}}}\Vert \Delta u_{n}\Vert _{2},
\end{align*}%
which implies that $\Vert \Delta u_{n}\Vert _{2}^{2}$ is bounded. Since $%
M_{\mu ,p}(a_{n})>0$ and $\Vert \nabla u_{n}\Vert _{2}^{2}<a_{n}\Vert \Delta
u_{n}\Vert _{2}$, there exists a constant $C>0$ such that $\Vert \nabla
u_{n}\Vert _{2}^{2}<C$. Then we obtain that
\begin{align*}
M_{\mu ,p}(a)& \leq \Psi _{\mu ,p}\left( \left( v_{n}\right) _{s_{n}}\right)
=\frac{\left( \frac{a}{a_{n}}\right) ^{2+4b}}{2}\Vert \Delta \left(
u_{n}\right) _{s_{n}}\Vert _{2}^{2}-\frac{\mu \left( \frac{a}{a_{n}}\right)
^{2+2b}}{2}\Vert \nabla \left( u_{n}\right) _{s_{n}}\Vert _{2}^{2}-\frac{%
\left( \frac{a}{a_{n}}\right) ^{2+4b}}{p}\Vert \left( u_{n}\right)
_{s_{n}}\Vert _{p}^{p} \\
& =\left( \frac{a}{a_{n}}\right) ^{2+4b}\left( \frac{1}{2}\Vert \Delta
\left( u_{n}\right) _{s_{n}}\Vert _{2}^{2}-\frac{\mu \left( \frac{a}{a_{n}}%
\right) ^{-2b}}{2}\Vert \nabla \left( u_{n}\right) _{s_{n}}\Vert _{2}^{2}-%
\frac{1}{p}\Vert \left( u_{n}\right) _{s_{n}}\Vert _{p}^{p}\right) \\
& \leq \left( \frac{a}{a_{n}}\right) ^{2+4b}\left( \Psi _{\mu ,p}\left(
(u_{n})_{s_{n}}\right) +\frac{\mu \left\vert 1-\left( \frac{a}{a_{n}}\right)
^{-2b}\right\vert }{2}s_{n}^{2}\Vert \nabla u_{n}\Vert _{2}^{2}\right) \\
& <\left( \frac{a}{a_{n}}\right) ^{2+4b}\left( \Psi _{\mu ,p}\left(
(u_{n})_{s_{n}}\right) +\frac{\mu C\left\vert 1-\left( \frac{a}{a_{n}}%
\right) ^{-2b}\right\vert }{2}\right) \\
& <\left( \frac{a}{a_{n}}\right) ^{2+4b}\left( M_{\mu ,p}(a_{n})+\frac{\mu
C\left\vert 1-\left( \frac{a}{a_{n}}\right) ^{-2b}\right\vert }{2}\right)
+\left( \frac{a}{a_{n}}\right) ^{2+4b}\epsilon .
\end{align*}%
By the arbitrariness of $\epsilon $, we deduce that
\begin{equation*}
M_{\mu ,p}(a)\leq \left( \frac{a}{a_{n}}\right) ^{2+4b}M_{\mu ,p}(a_{n})+%
\frac{\mu C\left\vert 1-\left( \frac{a}{a_{n}}\right) ^{-2b}\right\vert
\left( \frac{a}{a_{n}}\right) ^{2+4b}}{2}.
\end{equation*}%
Hence, we have $M_{\mu ,p}(a)\leq \liminf_{n\rightarrow \infty }M_{\mu
,p}(a_{n})$.

According to the definition of $M_{\mu ,p}(a)$, for any $\epsilon >0$
sufficiently small, there exists $u\in \mathcal{P}_{a}^{-}$ such that
\begin{equation*}
\Psi _{\mu ,p}(u)\leq M_{\mu ,p}(a)+\epsilon \ \text{and }\Psi _{\mu
,p}(u)>0.
\end{equation*}%
Set $w_{n}=\left( \frac{a_{n}}{a}\right) ^{c}u\left( \left( \frac{a_{n}}{a}%
\right) ^{b}\cdot \right) .$ Then there exists $t_{n}>0$ such that $\left(
w_{n}\right) _{t_{n}}\in \mathcal{P}_{a_{n}}^{-}$. Clearly, $%
t_{n}\rightarrow 1$ and $t_{n}$ is bounded. Since $Q_{p}(u)=0$, we calculate
that
\begin{align*}
\Psi _{\mu ,p}(u)& =\frac{p\gamma _{p}-2}{2p\gamma _{p}}\Vert \Delta u\Vert
_{2}^{2}-\frac{\mu (p\gamma _{p}-1)}{2p{\gamma _{p}}}\Vert \nabla u\Vert
_{2}^{2} \\
& \geq \frac{p\gamma _{p}-2}{2p\gamma _{p}}\Vert \Delta u\Vert _{2}^{2}-%
\frac{\mu a(p\gamma _{p}-1)}{2p{\gamma _{p}}}\Vert \Delta u\Vert _{2},
\end{align*}%
which implies that $\Vert \Delta u\Vert _{2}^{2}$ is bounded. Since $M_{\mu
,p}(a)>0$ and $\Vert \nabla u\Vert _{2}^{2}<a\Vert \Delta u\Vert _{2}$,
there exists a constant $C>0$ such that $\Vert \nabla u\Vert
_{2}^{2}<CM_{\mu ,p}(a)$. Then we have
\begin{align*}
M_{\mu ,p}(a_{n})& \leq \Psi _{\mu ,p}\left( \left( w_{n}\right)
_{s_{n}}\right) =\frac{\left( \frac{a_{n}}{a}\right) ^{2+4b}}{2}\Vert \Delta
u_{t_{n}}\Vert _{2}^{2}-\frac{\mu \left( \frac{a_{n}}{a}\right) ^{2+2b}}{2}%
\Vert \nabla u_{t_{n}}\Vert _{2}^{2}-\frac{\left( \frac{a_{n}}{a}\right)
^{2+4b}}{p}\Vert u_{t_{n}}\Vert _{p}^{p} \\
& =\left( \frac{a_{n}}{a}\right) ^{2+4b}\left( \frac{1}{2}\Vert \Delta
u_{t_{n}}\Vert _{2}^{2}-\frac{\mu \left( \frac{a_{n}}{a}\right) ^{-2b}}{2}%
\Vert \nabla u_{t_{n}}\Vert _{2}^{2}-\frac{1}{p}\Vert u_{t_{n}}\Vert
_{p}^{p}\right) \\
& =\left( \frac{a_{n}}{a}\right) ^{2+4b}\left( \frac{1}{2}\Vert \Delta
u_{t_{n}}\Vert _{2}^{2}-\frac{\mu }{2}\Vert \nabla u_{t_{n}}\Vert _{2}^{2}-%
\frac{1}{p}\Vert u_{t_{n}}\Vert _{p}^{p}\right) \\
& +\left( \frac{a_{n}}{a}\right) ^{2+4b}\frac{\mu \left( 1-\left( \frac{a_{n}%
}{a}\right) ^{-2b}\right) }{2}s_{n}^{2}\Vert \nabla u\Vert _{2}^{2} \\
& <\left( \frac{a_{n}}{a}\right) ^{2+4b}\Psi_{\mu,p}(u)+\left( \frac{a_{n}}{a%
}\right) ^{2+4b}\frac{\mu C\left\vert 1-\left( \frac{a_{n}}{a}\right)
^{-2b}\right\vert }{2}M_{\mu ,p}(a) \\
& <\left( \frac{a_{n}}{a}\right) ^{2+4b}\left( 1+\frac{\mu C\left\vert
1-\left( \frac{a_{n}}{a}\right) ^{-2b}\right\vert }{2}\right) M_{\mu
,p}(a)+\left( \frac{a_{n}}{a}\right) ^{2+4b}\epsilon .
\end{align*}%
By the arbitrariness of $\epsilon $, one has
\begin{equation*}
M_{\mu ,p}(a_{n})\leq \left( \frac{a_{n}}{a}\right) ^{2+4b}\left( 1+\frac{%
\mu C\left\vert 1-\left( \frac{a_{n}}{a}\right) ^{-2b}\right\vert }{2}\right)
M_{\mu ,p}(a),
\end{equation*}%
which implies that $\limsup_{n\rightarrow \infty }M_{\mu ,p}(a_{n})\leq
M_{\mu ,p}(a)$. Hence, we obtain that $M_{\mu ,p}(a_{n})\rightarrow M_{\mu
,p}(a)$ as $n\rightarrow \infty $.

$(ii)$ Let $\beta >1$. Suppose that for any $\epsilon >0$, there exists $%
u\in \mathcal{P}_{\beta ^{-1}a}^{-}$ such that
\begin{equation}
\Psi _{\mu ,p}(u)\leq M_{\mu ,p}(\beta ^{-1}a)+\epsilon .  \label{E39}
\end{equation}%
Let $v=\beta ^{c}u\left( \beta ^{b}x\right) $. Then there exists $s^{-}>0$
such that $v_{s^{-}}\in \mathcal{P}_{a}^{-}$. Clearly,
\begin{equation}
1+2b\left\{
\begin{array}{ll}
<0, & \ \text{if}\ \overline{p}<p<4^{\ast }, \\
=0, & \ \text{if}\ p=4^{\ast }.%
\end{array}%
\right.  \label{E40}
\end{equation}%
Then by Lemma \ref{T2.7} $(iv)$ and (\ref{E39}), we have
\begin{align*}
M_{\mu ,p}(a)\leq \Psi _{\mu ,p}(v_{s^{-}})=& \frac{\beta ^{2+4b}}{2}\Vert
\Delta u_{s^{-}}\Vert _{2}^{2}-\frac{\mu \beta ^{2+2b}}{2}\Vert \nabla
u_{s^{-}}\Vert _{2}^{2}-\frac{\beta ^{2+4b}}{p}\Vert u_{s^{-}}\Vert _{p}^{p}
\\
=& \beta ^{2+4b}\left( \frac{1}{2}\Vert \Delta u_{s^{-}}\Vert _{2}^{2}-\frac{%
\mu \beta ^{-2b}}{2}\Vert \nabla u_{s^{-}}\Vert _{2}^{2}-\frac{1}{p}\Vert
u_{s^{-}}\Vert _{p}^{p}\right) \\
<& \beta ^{2+4b}\left( \frac{1}{2}\Vert \Delta u_{s^{-}}\Vert _{2}^{2}-\frac{%
\mu }{2}\Vert \nabla u_{s^{-}}\Vert _{2}^{2}-\frac{1}{p}\Vert u_{s^{-}}\Vert
_{p}^{p}\right) \\
<& \beta ^{2+4b}\Psi_{\mu,p}(u) \\
<& \beta ^{2+4b}M_{\mu ,p}(\beta ^{-1}a)+\beta ^{2+4b}\epsilon .
\end{align*}%
By the arbitrariness of $\epsilon $, we have
\begin{equation}
M_{\mu ,p}(a)\leq \beta ^{2+4b}M_{\mu ,p}(\beta ^{-1}a),  \label{E41}
\end{equation}%
and the inequality is strict if $M_{\mu ,p}(\beta ^{-1}a)$ is attained. It
follows from (\ref{E40})-(\ref{E41}) that
\begin{equation}
M_{\mu ,p}(a)%
\begin{cases}
<M_{\mu ,p}(\beta ^{-1}a),\  & \text{for}\ \overline{p}<p<4^{\ast }, \\
\leq M_{\mu ,p}(\beta ^{-1}a),\  & \text{for}\ p=4^{\ast }.%
\end{cases}
\label{E42}
\end{equation}%
Hence, by (\ref{E41}) and (\ref{E42}), we easily reach the conclusion. The
proof is complete.
\end{proof}

\begin{lemma}
\label{T4.2} Let $\overline{p}<p<4^{\ast }$ and $N\geq 2$. Then there exists
$\bar{\mu}_{a}\in \left( 0,(C_{0}a^{2-p})^{\frac{1}{p\gamma _{p}-2}}\right) $
such that for any $\mu \in (0,\bar{\mu}_{a})$, there holds
\begin{equation}
M_{\mu ,p}(a)<M_{\mu ,p}(a_{1})+m_{\mu ,p}(a_{2})  \label{E43}
\end{equation}%
for any $a_{1},a_{2}>0$ satisfying $a_{1}^{2}+a_{2}^{2}=a^{2}.$
\end{lemma}

\begin{proof}
For any $a_{1},a_{2}>0$ satisfying $a_{1}^{2}+a_{2}^{2}=a^{2},$ it follow
from Lemmas \ref{T3.5} and \ref{T4.1} that%
\begin{equation*}
m_{\mu ,p}(a)\leq \left( \frac{a}{a_{2}}\right) ^{2}m_{\mu ,p}(a_{2})\text{
and }M_{\mu ,p}(a)\leq \left( \frac{a}{a_{1}}\right) ^{2+4b}M_{\mu
,p}(a_{1}),
\end{equation*}%
which indicates that
\begin{align}
M_{\mu ,p}(a)& =M_{\mu ,p}(a_{1})+m_{\mu ,p}(a_{2})+M_{\mu ,p}(a)-M_{\mu
,p}(a_{1})-m_{\mu ,p}(a_{2})  \notag \\
& \leq M_{\mu ,p}(a_{1})+m_{\mu ,p}(a_{2})+\left(1-\left( \frac{a_{1}}{a}%
\right) ^{2+4b}\right) M_{\mu ,p}(a)-\frac{a^{2}-(a_{1})^{2}}{a^{2}}m_{\mu
,p}(a).  \label{E44}
\end{align}%
Define the function $F_{\mu }:(0,a]\rightarrow \mathbb{R}$ by%
\begin{equation*}
F_{\mu }(x)=\left( 1-\left( \frac{x}{a}\right) ^{2+4b}\right) M_{\mu ,p}(a)-%
\frac{a^{2}-x^{2}}{a^{2}}m_{\mu ,p}(a)\text{ for }x>0.
\end{equation*}%
Clearly, there hold
\begin{equation}
\lim\limits_{x\rightarrow 0^{+}}F_{\mu }(x)=-\infty \ \text{and }F_{\mu
}(a)=0.  \label{E45}
\end{equation}%
Moreover, by calculating the derivative of $F_{\mu }(x)$ one has
\begin{equation*}
F_{\mu }^{\prime }(x)=-(2+4b)\frac{x^{1+4b}}{a^{2+4b}}M_{\mu ,p}(a)+\frac{2x%
}{a^{2}}m_{\mu ,p}(a).
\end{equation*}%
Since $2+4b<0$, we easily obtain that $F_{\mu }^{\prime }(x)$ is strictly
decreasing on $(0,a),$ $\lim\limits_{x\rightarrow 0^{+}}F_{\mu }^{\prime
}(x)=+\infty $ and
\begin{equation*}
F_{\mu }^{\prime }(a)=-\frac{2+4b}{a}M_{\mu ,p}(a)+\frac{2}{a}m_{\mu ,p}(a).
\end{equation*}

Next we claim that $F_{\mu }^{\prime }(a)\geq 0.$ Note that $M_{\mu ,p}(a)$
and $m_{\mu ,p}(a)$ are both decreasing on $\mu $. Then $F_{\mu }^{\prime
}(a)$\ is decreasing on $\mu .$ It follows from Corollary \ref{L3.5} that
\begin{equation*}
\lim\limits_{\mu \rightarrow 0^{+}}M_{\mu ,p}(a)=\frac{p\gamma _{p}-2}{%
2p\gamma _{p}}\left( \frac{1}{\mathcal{C}_{N,p}^{p}\gamma _{p}a^{p(1-\gamma
_{p})}}\right) ^{\frac{2}{p\gamma _{p}-2}}\text{ and}\ \lim\limits_{\mu
\rightarrow 0^{+}}m_{\mu ,p}(a)=0,
\end{equation*}%
which implies that $\lim\limits_{\mu \rightarrow 0^{+}}F_{\mu }^{\prime
}(a)>0.$ Let%
\begin{equation*}
\bar{\mu}_{a}=\sup \left\{ \mu \in \left( 0,(C_{0}a^{2-p})^{\frac{1}{p\gamma
_{p}-2}}\right) :F_{\mu }^{\prime }(a)>0\right\} .
\end{equation*}%
Then for any $\mu \in (0,\bar{\mu}_{a}),$ we have $F_{\mu }^{\prime }(a)\geq
0.$ So it follows that $F_{\mu }^{\prime }(x)>0$ for any $x\in (0,a),$ which
shows that $F_{\mu }(x)$ is strictly increasing on $(0,a)$. Using this,
together with (\ref{E45}), we have $F_{\mu }(x)<0$ for any $x\in (0,a)$.
Hence, by (\ref{E44}), we conclude that (\ref{E43}) holds. The proof is
complete.
\end{proof}

By Lemma \ref{T4.1}, we find that $M_{\mu ,4^{\ast }}(a)\leq M_{\mu ,4^{\ast
}}(\beta ^{-1}a)$ for any $\beta>1$. To show the inequality is strict, we
need the following lemma.

\begin{lemma}
\label{T4.3} Let $p=4^{\ast }$ and $N\geq 9.$ Assume that condition (\ref%
{E4}) holds. Then we have $M_{\mu ,4^{\ast }}(a)<m_{\mu ,4^{\ast }}(a)+%
\frac{2}{N}\mathcal{S}^{N/4}$.
\end{lemma}

\begin{proof}
Let $V_{\epsilon }=\chi (x)u_{\epsilon },$ where $u_{\epsilon }$ is
given as (\ref{E8}), and $\chi (x)$ is a cut-off function such that $\chi
(x)=1$ for $|x|\leq R$ and $\chi (x)=0$ for $|x|>R$, here $R>0$. By \cite%
{Gu1994}, we have the following estimates:%
\begin{equation*}
\begin{array}{l}
\Vert \Delta {V_{\epsilon }}\Vert _{2}^{2}=\mathcal{S}^{N/4}+O\left( \epsilon
^{\frac{N-4}{2}}\right) , \\
\Vert {V_{\epsilon }}\Vert _{4^{\ast }}^{4^{\ast }}=\mathcal{S}^{N/4}+O\left(
\epsilon ^{\frac{N}{2}}\right) , \\
\Vert {V_{\epsilon }}\Vert _{2}^{2}=C\epsilon ^{2}+O\left( \epsilon ^{\frac{N-4%
}{2}}\right) .%
\end{array}
\label{e3.10}
\end{equation*}%
Moreover, there holds
\begin{align}
\Vert \nabla {V_{\epsilon }}\Vert _{2}^{2}& =C\epsilon ^{\frac{N-4}{2}%
}\int_{B_{2R}(0)}\frac{|x|^{2}}{\left( \epsilon +|x|^{2}\right) ^{N-2}}%
dx+O\left( \epsilon ^{\frac{N-4}{2}}\right)  \notag \\
& =C\epsilon \int_{B_{2R/\epsilon }(0)}\frac{|x|^{2}}{\left(
1+|x|^{2}\right) ^{N-2}}dx+O\left( \epsilon ^{\frac{N-4}{2}}\right)  \notag
\\
& =C\epsilon +O\left( \epsilon ^{\frac{N-4}{2}}\right) .  \label{E46}
\end{align}%
A direct calculation shows that
\begin{align}
\int_{\mathbb{R}^{N}}V_{\epsilon }^{\frac{N+4}{N-4}}\psi dx& \geq
C\inf_{B_{R}(0)}\psi \int_{B_{2R}(0)}\epsilon ^{\frac{N+4}{4}}\left( \frac{1}{%
\epsilon +|x|^{2}}\right) ^{\frac{N+4}{2}}dx  \notag \\
& =C\epsilon ^{\frac{N-4}{4}}\int_{B_{2R/\sqrt{\epsilon }}(0)}\left( \frac{1}{%
1+|x|^{2}}\right) ^{\frac{N+4}{2}}dx  \notag \\
& \geq C\epsilon ^{\frac{N-4}{4}}+o(\epsilon ^{\frac{N-4}{4}})  \label{E47}
\end{align}%
and%
\begin{align*}
\int_{\mathbb{R}^{N}}\psi V_{\epsilon }dx& \leq C\int_{B_{2R}(0)}\epsilon ^{%
\frac{N-4}{4}}\left( \frac{1}{\epsilon +|x|^{2}}\right) ^{\frac{N-4}{2}}dx
\notag \\
& \leq C\epsilon ^{\frac{N+4}{4}}\int_{B_{2R/\sqrt{\epsilon }}(0)}\left( \frac{1%
}{1+|x|^{2}}\right) ^{\frac{N-4}{2}}dx  \notag \\
& \leq C\epsilon ^{\frac{N+4}{4}}\int_{0}^{\frac{2R}{\sqrt{\epsilon }}%
}\left( \frac{1}{1+r^{2}}\right) ^{\frac{N-4}{2}}r^{N-1}dr  \notag \\
& \leq CR^{4}\epsilon ^{\frac{N-4}{4}}+o(\epsilon ^{\frac{N-4}{4}})
\label{e3.13}
\end{align*}%
for any $0<\psi \in L^{\infty }(\mathbb{R}^{N})$. Define $\overline{V}%
_{\epsilon ,\ell }=u_{\mu }^{+}+\ell V_{\epsilon }$, where $\ell >0$ and $%
u_{\mu }^{+}\in \mathcal{P}_{a}^{+}$ is the ground state solution of problem
(\ref{E3}). We also define $\widetilde{V}_{\epsilon ,\ell }=\bar{\beta}^{%
\frac{N-4}{4}}\overline{V}_{\epsilon ,\ell }\left( \bar{\beta}^{\frac{1}{2}%
}x\right) ,$ where $\bar{\beta}=\frac{\Vert \overline{V}_{\epsilon ,\ell
}\Vert _{2}}{a}>1$. A direct calculation shows that
\begin{equation}
\begin{array}{ll}
\Vert \widetilde{V}_{\epsilon ,\ell }\Vert _{2}^{2}=\overline{\beta }%
^{-2}\Vert \overline{V}_{\epsilon ,\ell }\Vert _{2}^{2}, & \ \Vert \Delta
\widetilde{V}_{\epsilon ,\ell }\Vert _{2}^{2}=\Vert \Delta \overline{V}%
_{\epsilon ,\ell }\Vert _{2}^{2}, \\
\Vert \nabla \widetilde{V}_{\epsilon ,\ell }\Vert _{2}^{2}=\overline{\beta }%
^{-1}\Vert \nabla \overline{V}_{\epsilon ,\ell }\Vert _{2}^{2}, & \ \Vert
\widetilde{V}_{\epsilon ,\ell }\Vert _{4^{\ast }}^{4^{\ast }}=\Vert
\overline{V}_{\epsilon ,\ell }\Vert _{4^{\ast }}^{4^{\ast }}.%
\end{array}
\label{E48}
\end{equation}%
Clearly, $\widetilde{V}_{\epsilon ,\ell }\in S_{a}$. It follows from Lemma %
\ref{T2.7} that there exists $s_{\epsilon ,\ell }$ such that $\left(
\widetilde{V}_{\epsilon ,\ell }\right) _{s_{\epsilon ,\ell }}\in \mathcal{P}%
_{a}^{+}$. Then we have
\begin{equation}
2\left\Vert \Delta \left( \widetilde{V}_{\epsilon ,\ell }\right)
_{s_{\epsilon ,\ell }}\right\Vert _{2}^{2}=2s_{\epsilon ,\ell }^{4}\Vert
\Delta \widetilde{V}_{\epsilon ,\ell }\Vert _{2}^{2}=\mu s_{\epsilon ,\ell
}^{2}\Vert \nabla \overline{V}_{\epsilon ,\ell }\Vert _{2}^{2}+2s_{\epsilon
,\ell }^{24^{\ast }}\Vert \overline{V}_{\epsilon ,\ell }\Vert _{4^{\ast
}}^{4^{\ast }}.  \label{E49}
\end{equation}%
Since $u_{a}\in \mathcal{P}_{a}^{+}$, it follows from Lemma \ref{T2.7} that $%
s_{\epsilon ,0}>1$. Combining with (\ref{E48}) and (\ref{E49}), one has $%
s_{\epsilon ,\ell }\rightarrow 0$ as $\ell \rightarrow +\infty $ uniformly
for $\epsilon >0$ sufficiently small. From Lemma \ref{T2.7}, we deduce that $%
s_{\epsilon ,\ell }$ is unique. So it is standard to show that $s_{\epsilon
,\ell }$ is continuous on $\ell $, which shows that there exists $\ell
_{\epsilon }>0$ such that $s_{\epsilon ,\ell _{\epsilon }}=1$. Then we have
\begin{equation}
M_{\mu ,4^{\ast }}(a)\leq \sup\limits_{\ell \geq 0}\Psi_{\mu,4^{\ast}}\left(
\widetilde{V}_{\epsilon ,\ell }\right) .  \label{E50}
\end{equation}%
There exists $\ell _{0}>0$ such that
\begin{equation*}
\Psi_{\mu,4^{\ast}}\left( \widetilde{V}_{\epsilon ,\ell }\right) =\frac{1}{2}%
\Vert \Delta \overline{V}_{\epsilon ,\ell }\Vert _{2}^{2}-\frac{\mu \bar{%
\beta}^{-1}}{2}\Vert \nabla \overline{V}_{\epsilon ,\ell }\Vert _{2}^{2}-%
\frac{1}{4^{\ast }}\Vert \overline{V}_{\epsilon ,\ell }\Vert _{4^{\ast
}}^{4^{\ast }}<m_{\mu ,4^{\ast }}(a)+\frac{2}{N}\mathcal{S}^{\frac{N}{4}%
}-\delta
\end{equation*}%
for $\ell <\frac{1}{\ell _{0}}$ and $\ell >\ell _{0}$ with $\delta >0$. It
follows from (\ref{E46}) and (\ref{E47}) that
\begin{equation*}
\bar{\beta}^{2}=\frac{\Vert \overline{V}_{\epsilon ,\ell }\Vert _{2}^{2}}{%
a^{2}}=1+\frac{2\ell }{a^{2}}\int_{\mathbb{R}^{N}}u_{\mu }^{+}V_{\epsilon
}dx+\frac{\ell ^{2}}{a^{2}}\Vert V_{\epsilon }\Vert _{2}^{2}.
\end{equation*}%
Then we have
\begin{equation}
1-\bar{\beta}^{-1}=\frac{\bar{\beta}^{2}-1}{\bar{\beta}(\bar{\beta}+1)}\leq
C\left( \frac{2\ell }{a^{2}}\int_{\mathbb{R}^{N}}u_{\mu }^{+}V_{\epsilon }dx+%
\frac{\ell ^{2}}{a^{2}}\Vert V_{\epsilon }\Vert _{2}^{2}\right) .
\label{E51}
\end{equation}%
Note that $(1+x)^{4^{\ast }}-1-x^{4^{\ast }}-4^{\ast }x^{\frac{N+4}{N-4}%
}\geq 0$ for all $x\geq 0$, we deduce that%
\begin{equation}
-\frac{1}{4^{\ast }}\Vert \overline{V}_{\epsilon ,\ell }\Vert _{4^{\ast
}}^{4^{\ast }}\leq -\frac{1}{4^{\ast }}\Vert u_{a}\Vert _{4^{\ast
}}^{4^{\ast }}-\frac{1}{4^{\ast }}\Vert \ell V_{\epsilon }\Vert _{4^{\ast
}}^{4^{\ast }}-\int_{\mathbb{R}^{N}}\left( \ell V_{\epsilon }\right) ^{\frac{%
N+4}{N-4}}u_{\mu }^{+}dx.  \label{E52}
\end{equation}%
By (\ref{E48})-(\ref{E52}) and the fact that $u_{\mu }^{+}$ is the
solution of problem (\ref{E3}) for some $\lambda _{\mu }^{+}>0,$ for $\ell
<\frac{1}{\ell _{0}}$ one has%
\begin{align*}
\Psi_{\mu,4^{\ast}}\left( \widetilde{V}_{\epsilon ,\ell }\right) & =\frac{1}{2}%
\Vert \Delta \overline{V}_{\epsilon ,\ell }\Vert _{2}^{2}-\frac{\mu \bar{%
\beta}^{-1}}{2}\Vert \nabla \overline{V}_{\epsilon ,\ell }\Vert _{2}^{2}-%
\frac{1}{4^{\ast }}\Vert \overline{V}_{\epsilon ,\ell }\Vert _{4^{\ast
}}^{4^{\ast }} \\
& \leq m_{\mu ,4^{\ast }}(a)+\Psi_{\mu,4^{\ast}}(\ell V_{\epsilon })-\int_{\mathbb{R%
}^{N}}\left( \ell V_{\epsilon }\right) ^{\frac{N+4}{N-4}}u_{\mu }^{+}dx \\
& -\ell \lambda _{\mu }^{+}\int_{\mathbb{R}^{N}}u_{\mu }^{+}V_{\epsilon
}dx+\ell \int_{\mathbb{R}^{N}}(u_{\mu }^{+})^{\frac{N+4}{N-4}}V_{\epsilon
}dx+\frac{\mu (1-\bar{\beta}^{-1})}{2}\Vert \nabla \overline{V}_{\epsilon
,\ell }\Vert _{2}^{2} \\
& \leq m_{\mu ,4^{\ast }}(a)+\frac{2}{N}\mathcal{S}^{\frac{N}{4}%
}+C(R^{4}-1)\epsilon ^{\frac{N-4}{4}}+o(\epsilon ^{\frac{N-4}{4}}),
\end{align*}%
and thus
\begin{equation*}
\Psi_{\mu,4^{\ast}}\left( \widetilde{V}_{\epsilon ,\ell }\right) <m_{\mu ,4^{\ast
}}(a)+\frac{2}{N}\mathcal{S}^{\frac{N}{4}}
\end{equation*}%
by taking $R,\epsilon >0$ sufficiently small. Using this, together with (\ref%
{E50}), we have
\begin{equation*}
M_{\mu ,4^{\ast }}(a)\leq \sup\limits_{\ell \geq 0}\Psi_{\mu,4^{\ast}}\left(
\widetilde{V}_{\epsilon ,\ell }\right) <m_{\mu ,4^{\ast }}(a)+\frac{2}{N}%
\mathcal{S}^{\frac{N}{4}}.
\end{equation*}%
The proof is complete.
\end{proof}

\begin{lemma}
\label{T4.4} Let $p=4^{\ast }$ and $N\geq 9.$ Assume that condition (\ref%
{E4}) holds. Then for any $\beta >1,$ there exists a constant $0<K<\frac{%
\beta }{(\beta -1)\mu }$ such that
\begin{equation}
M_{\mu ,4^{\ast }}(a)\leq \left( 1-\frac{K(\beta -1)}{\beta }\mu \right)
^{2}M_{\mu ,4^{\ast }}(\beta ^{-1}a).  \label{E53}
\end{equation}%
Moreover, there exists $\tilde{\mu}_{a}\in \left( 0,(C_{0}a^{2-p})^{\frac{1}{%
p\gamma _{p}-2}}\right) $ such that for any $\mu \in (0,\tilde{\mu}_{a})$,
\begin{equation*}
M_{\mu ,4^{\ast }}(a)<M_{\mu ,4^{\ast }}(\beta ^{-1}a)+m_{\mu ,4^{\ast
}}\left( \sqrt{1-\beta ^{-2}}a\right) .
\end{equation*}
\end{lemma}

\begin{proof}
By Corollary \ref{L3.5}, one has $M_{\mu ,4^{\ast }}(a)>0$. Let $\beta >1.$
To prove (\ref{E53}), we assume that for any $\epsilon \in (0,1)$
satisfying $M_{\mu ,4^{\ast }}(\beta ^{-1}a)+\epsilon <\frac{2}{N}\mathcal{S}%
^{\frac{N}{4}}$, there exists $v\in \mathcal{P}_{\beta ^{-1}a}^{-}$ such
that $\Psi_{\mu,4^{\ast}}(v)\leq M_{\mu ,4^{\ast }}(\beta ^{-1}a)+\epsilon $.
Then by Lemma \ref{T2.7} $(iii)$, there holds $\Vert \Delta v\Vert
_{2}>\rho_{a}=\left( \frac{N}{N+4}\right) ^{\frac{N-4}{8}}\mathcal{S}^{%
\frac{N}{8}}$. Moreover, we have%
\begin{eqnarray*}
\frac{2}{N}\mathcal{S}^{\frac{N}{4}} &>&\Psi_{\mu,4^{\ast}}(v)=\frac{2}{N}\Vert
\Delta v\Vert _{2}^{2}-\frac{\mu (N+4)}{4N}\Vert \nabla v\Vert _{2}^{2} \\
&\geq &\frac{2}{N}\Vert \Delta v\Vert _{2}^{2}-\frac{(N+4)\mu \beta ^{-1}a}{%
4N}\Vert \Delta v\Vert _{2} \\
&\geq &\frac{2}{N}\Vert \Delta v\Vert _{2}^{2}-\frac{(N+4)\mu a}{4N}\Vert
\Delta v\Vert _{2},
\end{eqnarray*}%
which shows that there exists $\delta _{0}>0$ independent of $\epsilon
,\beta $ such that $\Vert \Delta v\Vert _{2}<\delta _{0}$. Now we claim that
there exists $\delta ^{\ast }>0$ independent of $\epsilon ,\beta $ such that
$\Vert \nabla v\Vert _{2}>\frac{\delta ^{\ast }}{\beta }$. Otherwise, there
exists $v_{n}\in \mathcal{P}_{\beta ^{-1}a}^{-}$ satisfying $\Psi_{\mu,4^{\ast}}(v_{n})\leq M_{\mu ,4^{\ast }}(\beta ^{-1}a)+\epsilon $ such that $\Vert
\nabla v_{n}\Vert _{2}\leq \frac{\delta _{n}}{\beta }\rightarrow 0$ as $%
n\rightarrow \infty $. By calculating, we have $\Psi_{\mu,4^{\ast}}(v_{n})\geq
\frac{2}{N}\mathcal{S}^{\frac{N}{4}}+o_{n}(1)$. Since $M_{\mu ,4^{\ast
}}(\beta ^{-1}a)+\epsilon <\frac{2}{N}\mathcal{S}^{\frac{N}{4}}$, we have $%
\Psi_{\mu,4^{\ast}}(v_{n})<\frac{2}{N}\mathcal{S}^{\frac{N}{4}}$. This is a
contradiction.

For $\hat{s}_{1}>0,$ we set $V=\left[ \beta ^{\frac{4-N}{4}}v\left( \beta ^{-%
\frac{1}{2}}x\right) \right] _{\hat{s}_{1}}\in \mathcal{P}%
_{a}^{-}$.  A direct calculation shows that $\Psi _{\mu ,p}(V)>0$ and
\begin{equation*}
\Vert V\Vert _{2}^{2}=\beta ^{2}\Vert v_{\hat{s}_{1}}\Vert _{2}^{2},\ \Vert
\Delta V\Vert _{2}^{2}=\Vert \Delta v_{\hat{s}_{1}}\Vert _{2}^{2},\ \Vert
\nabla V\Vert _{2}^{2}=\beta \Vert \nabla v_{\hat{s}_{1}}\Vert _{2}^{2}\
\text{and}\ \Vert V\Vert _{4^{\ast }}^{4^{\ast }}=\Vert v_{\hat{s}_{1}}\Vert
_{4^{\ast }}^{4^{\ast }}.
\end{equation*}%
Since $Q_{4^{\ast}}(v_{\hat{s}_{1}})=Q_{4^{\ast}}(V)+\mu (\beta -1)\Vert \nabla v_{\hat{s}%
_{1}}\Vert _{2}^{2}>0$ and $v\in \mathcal{P}_{\beta ^{-1}a}^{-}$, it follows
from Lemma \ref{T2.7} $(i)$ that $\hat{s}_{1}<1$. A direct calculation gives
that
\begin{align*}
0<\Psi_{\mu,4^{\ast}}(V)=& -\frac{1}{2}\Vert \Delta V\Vert _{2}^{2}+\frac{N+4}{2N}%
\Vert V\Vert _{4^{\ast }}^{4^{\ast }}  \notag \\
=& -\frac{\hat{s}_{1}^{4}}{2}\Vert \Delta v\Vert _{2}^{2}+\frac{(N+4)\hat{s}%
_{1}^{24^{\ast }}}{2N}\Vert v\Vert _{4^{\ast }}^{4^{\ast }}  \notag \\
=& \hat{s}_{1}^{4}\left( -\frac{1}{2}\Vert \Delta v\Vert _{2}^{2}+\frac{(N+4)%
\hat{s}_{1}^{24^{\ast }-4}}{2N}\Vert v\Vert _{4^{\ast }}^{4^{\ast }}\right)
\notag \\
<& \hat{s}_{1}^{4}\left( -\frac{1}{2}\Vert \Delta v\Vert _{2}^{2}+\frac{N+4}{%
2N}\Vert v\Vert _{4^{\ast }}^{4^{\ast }}\right)  \notag \\
=& \hat{s}_{1}^{4}\Psi_{\mu,4^{\ast}}(v).
\end{align*}%
Let $\hat{s}=s^{2}$ for $s>0$ and
\begin{align*}
\hat{f}(\hat{s})& =\frac{1}{2}\Vert \Delta v_{s}\Vert _{2}^{2}-\frac{\mu }{2}%
\Vert \nabla v_{s}\Vert _{2}^{2}-\frac{1}{4^{\ast }}\Vert v_{s}V\Vert
_{4^{\ast }}^{4^{\ast }} \\
& =\frac{\hat{s}^{2}}{2}\Vert \Delta v\Vert _{2}^{2}-\frac{\mu \hat{s}}{2}%
\Vert \nabla v\Vert _{2}^{2}-\frac{\hat{s}^{4^{\ast }}}{4^{\ast }}\Vert
v\Vert _{4^{\ast }}^{4^{\ast }} \\
& =A\hat{s}^{2}-\mu B\hat{s}-C\hat{s}^{4^{\ast }},
\end{align*}%
where $A:=\frac{1}{2}\Vert \Delta v\Vert _{2}^{2}$, $B:=\frac{1}{2}\Vert
\nabla v\Vert _{2}^{2}$ and $C:=\frac{1}{4^{\ast }}\Vert v\Vert _{4^{\ast
}}^{4^{\ast }}.$ Then we have%
\begin{equation*}
\begin{array}{l}
\hat{f}^{\prime }(\hat{s})=2A\hat{s}-\mu B-4^{\ast }C\hat{s}^{4^{\ast }-1},
\\
\hat{f}^{\prime \prime }(\hat{s})=2A-4^{\ast }(4^{\ast }-1)C\hat{s}^{4^{\ast
}-2}, \\
\hat{f}^{\prime \prime \prime }(\hat{s})=-4^{\ast }(4^{\ast }-1)(4^{\ast
}-2)C\hat{s}^{4^{\ast }-3}.%
\end{array}%
\end{equation*}%
Since $v\in \mathcal{P}_{\beta^{-1}a}$, there hold $\hat{f}^{\prime }(1)=0$ and $%
\hat{f}^{\prime \prime }(1)<0$. Note that $\hat{s}_{1}<1$ and $V\in \mathcal{%
P}_{a}^{-}.$ Then it follows from Lemma \ref{T2.7} that%
\begin{equation}
\hat{f}^{\prime \prime }(\hat{s}),\hat{f}^{\prime \prime \prime }(\hat{s})<0%
\text{ for any }\hat{s}\in \lbrack \hat{s}_{1}^{2},1].  \label{E54}
\end{equation}%
Define the function $\hat{F}_{\mu }:[0,1]\rightarrow \mathbb{R}$ by%
\begin{equation*}
\hat{F}_{\mu }(s)=\hat{f}^{\prime \prime }(1)s-\hat{f}^{\prime \prime
}(1)+\mu B.
\end{equation*}%
By (\ref{E54}), we have $\hat{F}_{\mu }(\hat{s}_{1}^{2})=\hat{f}^{\prime
\prime }(1)\hat{s}_{1}^{2}-\hat{f}^{\prime \prime }(1)+\mu B>\mu \beta B$,
which implies that%
\begin{equation*}
\hat{s}_{1}^{2}<1+\frac{\mu (\beta -1)}{\hat{f}^{\prime \prime }(1)}B.
\end{equation*}%
Note that
\begin{equation*}
\frac{B}{\hat{f}^{\prime \prime }(1)}=\frac{B}{2A-4^{\ast }\left( 4^{\ast
}-1\right) C}=\frac{B}{-2\left( 4^{\ast }-2\right) A+\mu B\left( 4^{\ast
}-1\right) }<-\frac{B}{2\left( 4^{\ast }-2\right) A}.
\end{equation*}
Then together with $\Vert \Delta v\Vert _{2}<\delta _{0}$ and $\Vert \nabla
v\Vert _{2}>\delta ^{\ast }$, there exists a constant $K>0$ independent of $%
\epsilon ,\beta $ such that $\frac{B}{\hat{f}^{\prime \prime }(1)}<{-\frac{K}{\beta}}$. Thus there hold $\hat{s}_{1}^{2}<1-\frac{K(\beta -1)}{\beta}\mu <1$
and
\begin{align*}
M_{\mu ,4^{\ast }}(a)& \leq \hat{s}_{1}^{4}M_{\mu ,4^{\ast }}(\beta
^{-1}a)+\epsilon \\
& <\left( 1-\frac{K(\beta -1)}{\beta }\mu \right) ^{2}M_{\mu ,4^{\ast
}}(\beta ^{-1}a)+\epsilon .
\end{align*}%
By the arbitrariness of $\epsilon $, we conclude that
\begin{equation*}
M_{\mu ,4^{\ast }}(a)\leq \left( 1-\frac{K(\beta -1)}{\beta }\mu \right)
^{2}M_{\mu ,4^{\ast }}(\beta ^{-1}a).
\end{equation*}%
Using this, together with Lemma \ref{T3.5}, yields
\begin{align}
M_{\mu ,4^{\ast }}(a)& =M_{\mu ,4^{\ast }}(\beta ^{-1}a)+m_{\mu ,4^{\ast
}}\left( \sqrt{1-\beta ^{-2}}a\right) +M_{\mu ,4^{\ast }}(a)-M_{\mu ,4^{\ast
}}(\beta ^{-1}a)-m_{\mu ,4^{\ast }}\left( \sqrt{1-\beta ^{-2}}a\right)
\notag \\
& \leq M_{\mu ,4^{\ast }}(\beta ^{-1}a)+m_{\mu ,4^{\ast }}\left( \sqrt{%
1-\beta ^{-2}}a\right)  \notag \\
& +\left[ 1-\left( 1-\frac{K(\beta -1)}{\beta }\mu \right) ^{-2}\right]
M_{\mu ,4^{\ast }}(a)-\frac{\beta ^{2}-1}{\beta ^{2}}m_{\mu ,4^{\ast }}(a)
\notag \\
& \leq M_{\mu ,4^{\ast }}(\beta ^{-1}a)+m_{\mu ,4^{\ast }}\left( \sqrt{%
1-\beta ^{-2}}a\right)  \notag \\
& +\left[ \left( 1-\frac{K(\beta -1)}{\beta }\mu \right) ^{2}-1\right]
M_{\mu ,4^{\ast }}(a)-\frac{\beta ^{2}-1}{\beta ^{2}}m_{\mu ,4^{\ast }}(a)
\notag \\
& =M_{\mu ,4^{\ast }}(\beta ^{-1}a)+m_{\mu ,4^{\ast }}\left( \sqrt{1-\beta
^{-2}}a\right)  \notag \\
& +\left[ \frac{K^{2}(\beta -1)^{2}\mu ^{2}}{\beta ^{2}}-\frac{2K(\beta
-1)\mu }{\beta }\right] M_{\mu ,4^{\ast }}(a)-\frac{\beta ^{2}-1}{\beta ^{2}}%
m_{\mu ,4^{\ast }}(a).  \label{E55}
\end{align}%
Let $H_{\mu }(s)=\left[ \frac{K^{2}(s-1)^{2}\mu ^{2}}{s^{2}}-\frac{%
2K(s-1)\mu }{s}\right] M_{\mu ,4^{\ast }}(a)-\frac{s^{2}-1}{s^{2}}m_{\mu
,4^{\ast }}(a)$ for $s\in \lbrack 1,\infty ).$ Then we have%
\begin{equation}
H_{\mu }(1)=0\text{ and }\lim_{s\rightarrow \infty }H_{\mu }(s)=(K^{2}\mu
^{2}-2K\mu )M_{\mu ,4^{\ast }}(a)-m_{\mu ,4^{\ast }}(a).  \label{E56}
\end{equation}%
By calculating the derivative of $H_{\mu }(s),$ we have%
\begin{eqnarray}
H_{\mu }^{\prime }(s) &=&\left[ \frac{2K^{2}(s-1)\mu ^{2}}{s^{2}}-\frac{%
2K^{2}(s-1)^{2}\mu ^{2}}{s^{3}}-\frac{2K\mu }{s}+\frac{2K(s-1)\mu }{s^{2}}%
\right] M_{\mu ,4^{\ast }}(a)  \notag \\
&&-\frac{2}{s^{3}}m_{\mu ,4^{\ast }}(a)  \notag \\
&=&\frac{1}{s^{3}}[2K\mu (K\mu (s-1)-s)M_{\mu ,4^{\ast }}(a)-2m_{\mu
,4^{\ast }}(a)]  \notag \\
&=&\frac{1}{s^{3}}\hat{H}_{\mu }(s),  \label{E57}
\end{eqnarray}%
where $\hat{H}_{\mu }(s):=2K\mu \lbrack K\mu (s-1)-s]M_{\mu ,4^{\ast
}}(a)-2m_{\mu ,4^{\ast }}(a)$ for $s>0.$ Since $0<1-\frac{K(\beta -1)}{\beta
}\mu <1$ for any $\beta >1,$ we obtain that $0<K\mu <1$ and $K\mu (s-1)-s<0$
for any $s>1,$ which implies that $\hat{H}_{\mu }(s)$ is decreasing on $s\in
\lbrack 1,\infty )$ and $\lim_{s\rightarrow \infty }\hat{H}_{\mu
}(s)=-\infty .$ Note that $M_{\mu ,4^{\ast }}(a)$ is decreasing on $\mu .$
Then there exists $C>0$ such that $M_{\mu ,4^{\ast }}(a)>C$ for any $\mu \in
\left( 0,(C_{0}a^{2-p})^{\frac{1}{p\gamma _{p}-2}}\right) .$ So we have%
\begin{eqnarray}
\hat{H}_{\mu }(1) &=&-2K\mu M_{\mu ,4^{\ast }}(a)-2m_{\mu ,4^{\ast }}(a)
\notag \\
&<&-2K\mu C-2m_{\mu ,4^{\ast }}(a)  \notag \\
&=&\mu \left[ -2KC-2\frac{m_{\mu ,4^{\ast }}(a)}{\mu }\right] .
\label{E58}
\end{eqnarray}%
It follows from Lemma \ref{T2.7} $(iii)$ that for any $u\in P_{\mu ,a}^{+}$,
there holds $\Vert \Delta u\Vert _{2}<\frac{\mu a(N+4)}{16},$ which
indicates that
\begin{equation}
\lim\limits_{\mu \rightarrow 0^{+}}\frac{m_{\mu ,4^{\ast }}(a)}{\mu }=0.
\label{E59}
\end{equation}%
Let
\begin{equation*}
\tilde{\mu}_{a}=\sup \left\{ \mu ^{\ast }\in \left( 0,(C_{0}a^{2-p})^{\frac{1%
}{p\gamma _{p}-2}}\right) :-2KC-2\frac{m_{\mu ,4^{\ast }}(a)}{\mu }\leq 0%
\text{ for each }\mu \in (0,\mu ^{\ast })\right\} .
\end{equation*}%
Then it follows from (\ref{E57})--(\ref{E59}) that $\hat{H}_{\mu }(s)<0$
for any $\mu \in (0,\tilde{\mu}_{a}).$ Therefore, by (\ref{E55})--(\ref%
{E56}) one has%
\begin{equation*}
M_{\mu ,4^{\ast }}(a)<M_{\mu ,4^{\ast }}(\beta ^{-1}a_{1})+m_{\mu ,4^{\ast
}}\left( \sqrt{1-\beta ^{-2}}a_{2}\right) \text{ for any }\beta >1.
\end{equation*}%
The proof is complete.
\end{proof}

Let
\begin{equation*}
\mu _{p,a}=\left\{
\begin{array}{ll}
\bar{\mu}_{a}, & \text{ if }\bar{p}<p<4^{\ast }, \\
\tilde{\mu}_{a}, & \text{ if }p=4^{\ast }.%
\end{array}%
\right.
\end{equation*}%
Then it follows from Lemmas \ref{T4.2} and \ref{T4.4} that for any $\mu \in
(0,\mu _{p,a})$, there holds
\begin{equation}
M_{\mu ,p}(a)<M_{\mu ,p}(a_{1})+m_{\mu ,p}(a_{2}),  \label{E60}
\end{equation}%
for any $a_{1},a_{2}>0\ $and$\ a_{1}^{2}+a_{2}^{2}=a^{2}.$

\begin{proposition}
\label{T4.5} Let $\bar{p}<p<4^{\ast }$ for $N\geq 2$, or $p=4^{\ast }$ for $%
N\geq 9$. Then for any $\mu \in (0,\mu _{p,a}),$ the minimization problem $%
M_{\mu ,p}(a)=\inf_{u\in \mathcal{P}_{a}^-}\Psi _{\mu ,p}(u)$ can be attained
by $u_{\mu , a}\in H^{2}(\mathbb{R}^{N})$. Moreover, problem (\ref{E3})
has a solution $u_{\mu , a}$ with some $\lambda _{\mu }^{-}>0$.
\end{proposition}

\begin{proof}
Let $\{u_{n}\}\subset \mathcal{P}_{a}^{-}$ be a minimizing sequence. Then
there hold $Q_{p}(u_{n})=0$ and $\Vert \Delta u_{n}\Vert _{2}>\rho_{0}$.
Since $Q_{p}(u_{n})=0$, we have
\begin{align*}
\Psi _{\mu ,p}(u_{n})& =\frac{p\gamma _{p}-2}{2p\gamma _{p}}\Vert \Delta
u_{n}\Vert _{2}^{2}-\frac{\mu (p\gamma _{p}-1)}{2p{\gamma _{p}}}\Vert \nabla
u_{n}\Vert _{2}^{2}+o_{n}(1) \\
& \geq \frac{p\gamma _{p}-2}{2p\gamma _{p}}\Vert \Delta u_{n}\Vert _{2}^{2}-%
\frac{\mu a(p\gamma _{p}-1)}{2p{\gamma _{p}}}\Vert \Delta u_{n}\Vert
_{2}+o_{n}(1),
\end{align*}%
which implies that $\{u_{n}\}$ is bounded in $H^{2}(\mathbb{R}^{N})$.

Next we claim that there exist a $d>0$ and a sequence $\{x_{n}\}\subset
\mathbb{R}^{N}$ such that
\begin{equation}
\int_{B_{r}(x_{n})}|u_{n}|^{p}dx\geq d>0\ \text{for some}\ r>0.
\label{E61}
\end{equation}%
Otherwise, we have $u_{n}\rightarrow 0$ in $L^{p}(\mathbb{R}^{N})$. This
indicates that%
\begin{equation*}
\begin{array}{l}
\Psi _{\mu ,p}(u_{n})=\frac{1}{2}\Vert \Delta u_{n}\Vert _{2}^{2}-\frac{\mu
}{2}\Vert \nabla u_{n}\Vert _{2}^{2}+o_{n}(1), \\
Q_{p}(u_{n})=2\Vert \Delta u_{n}\Vert _{2}^{2}-\mu \Vert \nabla u_{n}\Vert
_{2}^{2}+o_{n}(1),%
\end{array}
\end{equation*}%
which implies that $\Psi _{\mu ,p}(u_{n})=-\frac{1}{2}\Vert \Delta
u_{n}\Vert _{2}^{2}+o_{n}(1)$. This contradicts with $M_{\mu ,p}(a)>0$. Thus
there exist $d>0$ and $\{x_{n}\}\subset \mathbb{R}^{N}$ such that (\ref%
{E61}) holds, and so we can assume that%
\begin{equation*}
u_{n}(x-x_{n})\rightharpoonup u_{\mu ,a}\neq 0.
\end{equation*}%
Let $w_{n}=u_{n}(x-x_{n})-u_{\mu ,a}$. It follows from Brezis-Lieb lemma that%
\begin{equation*}
\begin{array}{l}
\Vert u_{n}\Vert _{2}^{2}=\Vert u_{\mu ,a}\Vert _{2}^{2}+\Vert w_{n}\Vert
_{2}^{2}+o_{n}(1), \\
\Vert \Delta u_{n}\Vert _{2}^{2}=\Vert \Delta u_{\mu ,a}\Vert _{2}^{2}+\Vert
\Delta w_{n}\Vert _{2}^{2}+o_{n}(1), \\
\Psi _{\mu ,p}(u_{n})=\Psi _{\mu ,p}(u_{n}(x-x_{n}))=\Psi _{\mu
,p}(w_{n})+\Psi _{\mu ,p}(u_{\mu ,a})+o_{n}(1).%
\end{array}%
\end{equation*}

Now, we prove that $w_{n}\rightarrow 0$ in $H^{2}(\mathbb{R}^{N})$. For
convenience, let $\Vert u_{\mu ,a}\Vert _{2}^{2}=a_{1}^{2}$. According to
Lemma \ref{T2.7} $(ii)$, there exists $s_{0}>0$ such that $(u_{\mu
,a})_{s_{0}}\in \mathcal{P}_{a_{1}}^{-}$. It follows from the Brezis-Lieb
Lemma and Lemma \ref{T2.7} $(iv)$ that
\begin{equation*}
\Psi _{\mu ,p}(u_{n})=\Psi _{\mu ,p}\left( w_{n}\right) +\Psi _{\mu
,p}\left( u_{\mu ,a}\right) +o_{n}(1)
\end{equation*}%
and
\begin{align}
\Psi _{\mu ,p}(u_{n})\geq & \Psi _{\mu ,p}\left( (u_{n})_{s_{0}}\right)
=\Psi _{\mu ,p}\left( (w_{n})_{s_{0}}\right) +\Psi _{\mu ,p}\left( (u_{\mu
,a})_{s_{0}}\right) +o_{n}(1)  \notag \\
\geq & \Psi _{\mu ,p}\left( (w_{n})_{s_{0}}\right) +M_{\mu
,p}(a_{1})+o_{n}(1).  \label{E62}
\end{align}%
We first claim that $a_{1}=a$. Otherwise, we assume that $a_{1}<a$. Then we
consider the following two cases.

Case $(i):\Psi _{\mu ,p}\left( (w_{n})_{s_{0}}\right) \geq m_{\mu ,p}\left(
\Vert w_{n}\Vert _{2}\right) $ up to a subsequence.\textbf{\ }It follows
from (\ref{E62}) that
\begin{align}
\Psi _{\mu ,p}(u_{n})& \geq \Psi _{\mu ,p}\left( (w_{n})_{s_{0}}\right)
+\Psi _{\mu ,p}\left( (u_{\mu ,a})_{s_{0}}\right) +o_{n}(1)  \notag \\
& \geq m_{\mu ,p}\left( \Vert w_{n}\Vert _{2}\right) +M_{\mu ,p}\left(
a_{1}\right) +o_{n}(1).  \label{E63}
\end{align}%
Considering the continuity of $m_{\mu ,p}(a)$ for any $a\in (0,a_{0})$, we
obtain that
\begin{equation*}
m_{\mu ,p}\left( \Vert w_{n}\Vert _{2}\right) +o_{n}(1)=m_{\mu ,p}\left(
\sqrt{a^{2}-a_{1}^{2}}\right) .
\end{equation*}%
Using (\ref{E63}), one has
\begin{equation*}
M_{\mu ,p}(a)\geq m_{\mu ,p}\left( \sqrt{a^{2}-a_{1}^{2}}\right) +M_{\mu
,p}\left( a_{1}\right) ,
\end{equation*}%
which contradicts with (\ref{E60}). So, $a_{1}=a$.

Case $(ii):\Psi _{\mu ,p}\left( (w_{n})_{s_{0}}\right) <m_{\mu ,p}\left(
\Vert w_{n}\Vert _{2}\right) $ for any $n$. Since $\Psi _{\mu ,p}\left(
(w_{n})_{s_{0}}\right) <m_{\mu ,p}\left( \Vert w_{n}\Vert _{2}\right) $ for
any $n$, by Lemma \ref{T2.7} $(ii)$, there exists $s_{n}\leq 1$ such that $%
(w_{n})_{s_{0}s_{n}}\in \mathcal{P}_{\Vert w_{n}\Vert _{2}}^{-}$. Then it
follows from Lemma \ref{T2.7} that $\Psi _{\mu ,p}\left( (u_{\mu
,a})_{s_{0}s_{n}}\right) \geq m_{\mu ,p}(a_{1}),$ and thus
\begin{align}
M_{\mu ,p}(a)& =\Psi _{\mu ,p}(u_{n})+o_{n}(1)  \notag \\
& \geq \Psi _{\mu ,p}\left( (u_{n})_{s_{0}s_{n}}\right) +o_{n}(1)  \notag \\
& =\Psi _{\mu ,p}\left( (w_{n})_{s_{0}s_{n}}\right) +\Psi _{\mu ,p}\left(
(u_{\mu ,a})_{s_{0}s_{n}}\right) +o_{n}(1)  \notag \\
& \geq m_{\mu ,p}(a_{1})+M_{\mu ,p}(\Vert w_{n}\Vert _{2})+o_{n}(1).
\label{E64}
\end{align}%
It follows from the continuity of $M_{\mu }(a)$ that $M_{\mu ,p}(\Vert
w_{n}\Vert _{2})=M_{\mu ,p}\left( \sqrt{a^{2}-a_{1}^{2}}\right) +o_{n}(1)$.
Using (\ref{E64}), we have
\begin{equation*}
M_{\mu ,p}(a)\geq m_{\mu ,p}(a_{1})+M_{\mu ,p}\left( \sqrt{a^{2}-a_{1}^{2}}%
\right) +o_{n}(1),
\end{equation*}%
which is impossible. So, $a_{1}=a$.

When $p<4^{\ast }$, since $a=a_{1}$, we have
\begin{equation*}
\Psi _{\mu ,p}\left( (w_{n})_{s_{0}}\right) =\frac{1}{2}\Vert \Delta
(w_{n})_{s_{0}}\Vert _{2}^{2}+o_{n}(1),
\end{equation*}%
together with (\ref{E62}), one has
\begin{equation*}
M_{\mu ,p}(a)+o_{n}(1)\geq \frac{1}{2}\Vert \Delta (w_{n})_{s_{0}}\Vert
_{2}^{2}+M_{\mu ,p}(a)+o_{n}(1),
\end{equation*}%
leading to $\Vert \Delta w_{n}\Vert _{2}^{2}=o_{n}(1)$.

When $p=4^{\ast }$, we have%
\begin{equation}
\begin{array}{l}
\Psi _{\mu ,4^{\ast}}\left( (w_{n})_{s_{0}}\right) =\frac{1}{2}\Vert \Delta
(w_{n})_{s_{0}}\Vert _{2}^{2}-\frac{1}{4^{\ast }}\Vert (w_{n})_{s_{0}}\Vert
_{4^{\ast }}^{4^{\ast }}+o_{n}(1), \\
Q_{4^{\ast}}\left( (w_{n})_{s_{0}}\right) =2\Vert \Delta (w_{n})_{s_{0}}\Vert
_{2}^{2}-2\Vert (w_{n})_{s_{0}}\Vert _{4^{\ast }}^{4^{\ast }}+o_{n}(1).%
\end{array}
\label{E65}
\end{equation}%
If $\Vert w_{n}\Vert _{4^{\ast }}^{4^{\ast }}=o_{n}(1)$, then it follows
from (\ref{E65}) that $\Psi _{\mu ,p}\left( (w_{n})_{s_{0}}\right) =\frac{1%
}{2}\Vert \Delta (w_{n})_{s_{0}}\Vert _{2}^{2}+o_{n}(1)$. Combining with (%
\ref{E62}), one has
\begin{equation*}
M_{\mu ,p}(a)+o_{n}(1)\geq \frac{1}{2}\Vert \Delta (w_{n})_{s_{0}}\Vert
_{2}^{2}+M_{\mu ,p}(a)+o_{n}(1),
\end{equation*}%
which shows that $\Vert \Delta w_{n}\Vert _{2}^{2}=o_{n}(1)$. Otherwise, we
can define $t_{n}=\left( \frac{\Vert \Delta (w_{n})_{s_{0}}\Vert _{2}^{2}}{%
\Vert (w_{n})_{s_{0}}\Vert _{4^{\ast }}^{4^{\ast }}}\right) ^{\frac{1}{%
2(4^{\ast }-2)}}$. Then there hold%
\begin{equation}
\begin{array}{l}
Q_{4^{\ast}}\left( (w_{n})_{s_{0}t_{n}}\right) =2\Vert \Delta (w_{n})_{s_{0}t_{n}}\Vert
_{2}^{2}-2\Vert (w_{n})_{s_{0}t_{n}}\Vert _{4^{\ast }}^{4^{\ast }}=o_{n}(1),
\\
\Psi _{\mu ,4^{\ast}}\left( (w_{n})_{s_{0}t_{n}}\right) =\frac{1}{2}\Vert \Delta
(w_{n})_{s_{0}t_{n}}\Vert _{2}^{2}-\frac{1}{4^{\ast }}\Vert
(w_{n})_{s_{0}t_{n}}\Vert _{4^{\ast }}^{4^{\ast }}+o_{n}(1)\geq \frac{2}{N}%
\mathcal{S}^{\frac{N}{4}}+o_{n}(1).%
\end{array}
\label{E66}
\end{equation}%
Now we claim that $t_{n}\leq 1$ up to a subsequence. Otherwise, there holds $%
t_{n}>1$ for all $n$. Then by (\ref{E65})-(\ref{E66}), we deduce that $%
\Psi _{\mu ,4^{\ast}}\left( (w_{n})_{s_{0}}\right) \geq 0$ and $Q_{4^{\ast}}\left(
(w_{n})_{s_{0}}\right) \geq 0$, together with (\ref{E62}), we obtain that $%
\Psi _{\mu ,4^{\ast}}\left( (w_{n})_{s_{0}}\right) =o_{n}(1)$ and thus $\Vert
\Delta w_{n}\Vert _{2}=o_{n}(1)$, which
is impossible. Hence, $t_{n}\leq 1$ up to a subsequence. Without loss of
generality, we may assume that $t_{n}\leq 1$ for any $n$. According to Lemma %
\ref{T2.7}, we obtain that $\Psi _{\mu ,4^{\ast}}\left( (u_{\mu
,a})_{s_{0}t_{n}}\right) \geq m_{\mu ,4^{\ast }}(a)$. Then by (\ref{E66}),
we have
\begin{align*}
M_{\mu ,4^{\ast }}(a)=\Psi_{\mu,4^{\ast}}(u_{n})+o_{n}(1)\geq & \Psi_{\mu,4^{\ast}}\left( (w_{n})_{s_{0}t_{n}}\right) +\Psi_{\mu,4^{\ast}}\left( (u_{\mu
,a})_{s_{0}t_{n}}\right) +o_{n}(1) \\
\geq & m_{\mu ,4^{\ast }}(a)+\frac{2}{N}\mathcal{S}^{\frac{N}{4}}+o_{n}(1),
\end{align*}%
which contradicts with $M_{\mu ,4^{\ast }}(a)<m_{\mu ,4^{\ast }}(a)+\frac{2}{%
N}\mathcal{S}^{\frac{N}{4}}$. Hence, we conclude that $M_{\mu
,p}(a)=\inf_{u\in \mathcal{P}_{a}^{-}}\Psi _{\mu ,p}(u)$ can be attained by $%
u_{\mu , a}$. Then there exists $\lambda _{\mu }^{-}\in \mathbb{R}$ such
that $(\lambda _{\mu }^{-},u_{\mu , a})$ is the solution of problem (\ref%
{E2})-(\ref{E3}). Using the Pohozaev identity, we easily obtain $\lambda
_{\mu }^{-}>0$. The proof is complete.
\end{proof}

To study the asymptotical property of $M_{\mu ,p}(a)$ as $\mu \rightarrow
0^{+}$, we need consider the following problem%
\begin{equation}
\left\{
\begin{array}{ll}
\Delta ^{2}u+\lambda u=|u|^{p-2}u, & \text{in }\mathbb{R}^{N}, \\
\int_{\mathbb{R}^{N}}|u|^{2}=a^{2}, &
\end{array}%
\right.  \label{E67}
\end{equation}%
where the associated energy functional is $\Psi _{0}(u)=\Psi _{\mu ,p}(u)$
with $\mu =0$ and the Pohozaev identity is $Q_{p}^{0}(u)=Q_{p}(u)$ with $\mu =0$.
Similar to Lemma \ref{T2.7}, Corollary \ref{L3.5} and Lemma \ref{T4.1} $%
(i)-(ii)$, we have the following results.

\begin{lemma}
\label{T4.6} Let $\bar{p}<p<4^{\ast }$ for $N\geq 2$, or $p=4^{\ast }$ for $%
N\geq 9$. Then for every $u\in S_{a}$, the function $\phi _{u}^{0}(t)=\phi
_{u}(t)$ with $\mu =0$ has a unique critical point $t_{u}>0$. Moreover, the
following statements hold.\newline
$(i)$ $(\phi _{u}^{0})^{\prime }(t_{u})=0$, $(\phi _{u}^{0})^{\prime }(t)>0$
for every $t\in (0,t_{u})$ and $(\phi _{u}^{0})^{\prime }(t)<0$ for every $%
t\in (t_{u},\infty )$.\newline
$(ii)$ $u_{t_{u}}\in \mathcal{P}_{0,a}^{-}$ and $\mathcal{P}_{0,a}=\mathcal{P%
}_{0,a}^{-}$, where $\mathcal{P}_{0,a}=\mathcal{P}_{a}$ and $\mathcal{P}%
_{0,a}^{-}=\mathcal{P}_{a}^{-}$ when $\mu =0.$\newline
$(iii)$ $\Vert \Delta u_{t_{u}}\Vert _{2}>\rho_{a}.$\newline
$(iv)$
\begin{equation*}
\Psi _{0}(u_{t_{u}})=\max \left\{ \Psi_{0}(u_{t}):t\geq 0\right\} \geq
\lim_{\mu\rightarrow0^+}g_{p}(s_{2})=\frac{p\gamma _{p}-2}{2p\gamma _{p}}\left( \frac{1}{\mathcal{C}%
_{N,p}^{p}\gamma _{p}a^{p(1-\gamma _{p})}}\right) ^{\frac{2}{p\gamma _{p}-2}%
}>0,
\end{equation*}%
where $g_{p}(s_{2})$ is given in Lemma \ref{T2.4}.\newline
$(v)$ The map $u\in S_{a}\mapsto t_{u}$ is of class $C^{1}$.
\end{lemma}

\begin{corollary}
\label{T4.7} Let $\bar{p}<p<4^{\ast }$ for $N\geq 2$, or $p=4^{\ast }$ for $%
N\geq 9$. Then we have $M_{p}(a):=\inf_{u\in \mathcal{P}_{0,a}^{-}}\Psi
_{0}(u)>0$.
\end{corollary}

\begin{lemma}
\label{T4.8} Let $\bar{p}<p<4^{\ast }$ for $N\geq 2$, or $p=4^{\ast }$ for $%
N\geq 9$. Then the following statements are true.\newline
$(i)$ The mapping $M_{p}(a):a\mapsto M_{p}(a)$ is continuous.\newline
$(ii)$
\begin{equation*}
M_{p}(\beta a)=\beta ^{2+4b}M_{p}(a)\leq M_{p}(a),\ \forall \beta {>1},
\end{equation*}%
where $b=\frac{2(p-2)}{8-N(p-2)}.$
\end{lemma}

\textbf{We are ready to prove Theorem \ref{T1.4}: }$(i)$\textbf{$\mathbf{\ }$%
}According to Proposition \ref{T4.5}, problem (\ref{E3}) has the second
solution $(\lambda _{\mu }^{-},u_{\mu }^{-})\in \mathbb{R}^{+}\times H^{2}(%
\mathbb{R}^{N}).$

$(ii)$ Following the idea of \cite[Theorem 3.7]{B2018}, the solution is
sign-changing.

$(iii)$\textbf{\ }For any $\mu _{n}>0$ satisfying $\mu _{n}\rightarrow 0^{+}$%
, we assume that $(\lambda _{n},u_{n})\in \left( \mathbb{R}^{+},\mathcal{P}%
_{a}^{-}\right) $ be the solution of problem (\ref{E3}) satisfying $\Psi
_{\mu _{n},p}(u_{n})=M_{\mu _{n},p}(a)>0$. Since
\begin{align*}
M_{p}(a)\geq M_{\mu _{n},p}(a)=& \frac{p\gamma _{p}-2}{2p\gamma _{p}}\Vert
\Delta u_{n}\Vert _{2}^{2}-\frac{p\gamma _{p}-1}{2p\gamma _{p}}\mu _{n}\Vert
\nabla u_{n}\Vert _{2}^{2} \\
\geq & \frac{p\gamma _{p}-2}{2p\gamma _{p}}\Vert \Delta u_{n}\Vert _{2}^{2}-%
\frac{p\gamma _{p}-1}{2p\gamma _{p}}\mu _{n}\Vert \nabla u_{n}\Vert _{2}^{2},
\end{align*}%
$\{u_{n}\}$ is bounded in $H^{2}(\mathbb{R}^{N})$. Since $(\lambda
_{n},u_{n})$ is the solution of problem (\ref{E3}), we have
\begin{equation*}
\Delta ^{2}u_{n}+\mu _{n}\Delta u_{n}+\lambda _{n}u_{n}=|u_{n}|^{p-2}u_{n}.
\end{equation*}%
Then there holds $\lambda _{n}a^{2}=-\Vert \Delta u_{n}\Vert _{2}^{2}+\mu
_{n}\Vert \nabla u_{n}\Vert _{2}^{2}+\Vert u_{n}\Vert _{p}^{p}$, which
implies that $\lambda _{n}$ is bounded. Hence, there exists $\lambda ^{-}\in
\mathbb{R}$ such that up to a subsequence, $\lim\limits_{n\rightarrow \infty
}\lambda _{n}=\lambda ^{-}$.

Now we claim that there exist a $d>0$ and a sequence $\{x_{n}\}\subset
\mathbb{R}^{N}$ such that
\begin{equation*}
\int_{B_{r}(x_{n})}|u_{n}|^{p}\geq d>0\ \text{for some}\ r>0.  \label{e3.33}
\end{equation*}%
Otherwise, we have $u_{n}\rightarrow 0$ in $L^{p}(\mathbb{R}^{N})$, which
implies that%
\begin{equation}
\begin{array}{l}
\Psi _{\mu _{n},p}(u_{n})=\frac{1}{2}\Vert \Delta u_{n}\Vert
_{2}^{2}+o_{n}(1), \\
Q_{p}(u_{n})=2\Vert \Delta u_{n}\Vert _{2}^{2}+o_{n}(1)=o_{n}(1).%
\end{array}
\label{E68}
\end{equation}%
This indicates that $\Psi _{\mu _{n},p}(u_{n})=o_{n}(1)$, which contradicts
with $M_{p}(a)>0$. Thus there exist $d>0$ and $\{x_{n}\}\subset \mathbb{R}%
^{N}$ such that (\ref{E68}) holds, and we can assume that%
\begin{equation*}
u_{n}(x-x_{n})\rightharpoonup u^{-}\neq 0.
\end{equation*}%
Furthermore, $(\lambda ^{-},u^{-})$ is the solution of problem (\ref{E67}%
). So we have $Q_{p}^{0}(u^{-})=0$. Let $v_{n}=u_{n}(x-x_{n})-u^{-}$. Next we
claim $v_{n}\rightarrow 0$ in $H^{2}(\mathbb{R}^{N})$. For convince, let $%
\Vert u^{-}\Vert _{2}^{2}=a_{1}^{2}\leq a^{2}$. Since $Q_{p}^{0}(u^{-})=0$,
using the Brezis-Lieb Lemma and Lemma \ref{T4.8} $(ii)$, we have
\begin{align}
M_{p}(a_{1})\leq \Psi _{0}(u^{-})=& \Psi _{\mu _{n},p}\left( u_{n}\right)
-\Psi _{0}\left( v_{n}\right) +o_{n}(1)  \notag \\
\leq & M_{p}(a)-\Psi _{0}\left( v_{n}\right) +o_{n}(1)  \notag \\
\leq & M_{p}(a_{1})-\Psi _{0}\left( v_{n}\right) +o_{n}(1).  \label{E69}
\end{align}%
Recall that $Q_{p}^{0}(u^{-})=0$, so we have $%
Q_{p}^{0}(v_{n})=Q_{p}(u_{n})-Q_{p}^{0}(u^{-})+o_{n}(1)=o_{n}(1)$. Then it follows from
Lemma \ref{T4.6} that $\Vert \Delta v_{n}\Vert _{2}=o_{n}(1)$ or $\Psi
_{0}(v_{n})>0$. If $\Psi _{0}(v_{n})>0$, then by (\ref{E69}), we have $%
M_{p}(a_{1})<M_{p}(a_{1}),$ which is impossible. Hence, there holds $\Vert
\Delta v_{n}\Vert _{2}=o_{n}(1).$ Using (\ref{E69}) again, we have $M_{\mu
,p}(a_{1})=\Psi _{0}(u^{-})=M_{p}(a),$ which implies that $M_{p}(a)$ is
attained at $u^{-}$ and $a_{1}=a.$

$(iv)$ When $p=4^{\ast }$, we know that problem (\ref{E67}) has a unique
solution $(0,u^{-})=(0,u_{\epsilon })$, where $u_{\epsilon }$ is given in (%
\ref{E8}). Hence, we have $M_{4^{\ast }}(a)=\frac{2}{N}\mathcal{S}^{\frac{N%
}{4}}$ and $\Vert \Delta u^{-}\Vert_{2}^{2}=\mathcal{S}^{\frac{N}{4}}$. The proof
is complete.

\section{Least action characterization}

\begin{lemma}
\label{T5.1} Let $\bar{p}<p<4^{\ast }$ and $N\geq 9$. Assume that condition
(\ref{E4}) holds. Then we have $\lim\limits_{p\rightarrow {(4^{\ast })^{-}}}\alpha
_{p}(a)=\alpha _{4^{\ast }}(a).$
\end{lemma}

\begin{proof}
We first claim that $\limsup_{p_{n}\rightarrow (4^{\ast })^{-}}\alpha
_{p_{n}}(a)\leq \alpha _{4^{\ast }}(a)$. It follows from Theorem \ref{T1.2}
that there exists $v\in \mathcal{M}_{a}$ such that
\begin{equation}
\Psi _{\mu ,4^{\ast }}(v)=\alpha _{4^{\ast }}(a).  \label{E70}
\end{equation}%
By the Lebesgue dominated convergence theorem and the Young inequality, we
have
\begin{equation*}
\Vert u\Vert _{q}^{q}\leq \frac{4^{\ast }-q}{4^{\ast }-2}\Vert u\Vert
_{2}^{2}+\frac{q-2}{4^{\ast }-2}\Vert u\Vert _{4^{\ast }}^{4^{\ast }}.
\end{equation*}%
Note that $\Vert u\Vert _{q}^{q}$ is continuous on $q\in \left( 2,4^{\ast }%
\right] $. Then we have
\begin{equation}
\Psi _{\mu ,p_{n}}(u)=\Psi _{\mu ,4^{\ast }}(u)+o_{n}(1)\ \text{and }%
Q_{p_{n}}(u)=Q_{4^{\ast }}(u)+o_{n}(1)=o_{n}(1).  \label{E71}
\end{equation}%
From (\ref{E70}), (\ref{E71}) and Corollary \ref{L3.5}, we deduce that
\begin{equation*}
\alpha _{p_{n}}(a)=\inf\limits_{u\in \mathcal{M}_{a}}\Psi _{\mu
,p_{n}}(u)\leq \Psi _{\mu ,p_{n}}(v)=\alpha _{4^{\ast }}(a)+o_{n}(1),
\end{equation*}%
which implies that $\limsup_{p_{n}\rightarrow (4^{\ast })^{-}}\alpha
_{p_{n}}(a)\leq \alpha _{4^{\ast }}(a)$. Similarly, we also have $%
\liminf_{p_{n}\rightarrow (4^{\ast })^{-}}\alpha _{p_{n}}(a)\geq \alpha
_{4^{\ast }}(a)$. The proof is complete.
\end{proof}

\begin{lemma}
\label{T5.2} Let $\bar{p}<p<4^{\ast }$ and $N\geq 9$. Assume that condition
(\ref{E4}) holds. If $\lim\limits_{p\rightarrow (4^{\ast })^{-}}\lambda
_{p}=\lambda _{4^{\ast }}$, then $\lim\limits_{p\rightarrow (4^{\ast
})^{-}}\theta _{\lambda _{p}}=\theta _{\lambda _{4^{\ast }}}$, where $\theta
_{\lambda }:=\inf \{I_{\lambda }(v):v\in H^{2}(\mathbb{R}^{N})\backslash
\{0\}$ and $I_{\lambda }^{\prime }(v)=0\}$
\end{lemma}

\begin{proof}
Without loss of generality, we assume that $\lim\limits_{p_{n}\rightarrow
(4^{\ast })^{-}}\lambda _{p_{n}}=\lambda _{4^{\ast }}$. We first claim that $%
\limsup_{p_{n}\rightarrow (4^{\ast })^{-}}\theta _{\lambda _{p_{n}}}\leq
\theta _{\lambda _{4^{\ast }}}$. According to the definition of $\theta
_{\lambda _{4^{\ast }}}$, for any $\epsilon >0$, there exists $u\in H^{2}(%
\mathbb{R}^{N})\backslash \{0\}$ such that $I_{\lambda _{4^{\ast }}}(u)\leq
\theta _{\lambda _{4^{\ast }}}+\epsilon $. By the continuity of $\Vert
u\Vert _{q}^{q}$ for $2<q\leq 4^{\ast }$ by Lemma \ref{T5.1} and $%
\lim\limits_{p_{n}\rightarrow (4^{\ast })^{-}}\lambda _{p_{n}}=\lambda
_{4^{\ast }}$, one has
\begin{equation*}
\Vert \Delta u\Vert _{2}^{2}-\mu \Vert \nabla u\Vert _{2}^{2}+\lambda
_{p_{n}}\Vert u\Vert _{2}^{2}-\Vert u\Vert _{p_{n}}^{p_{n}}=\Vert \Delta
u\Vert _{2}^{2}-\mu \Vert \nabla u\Vert _{2}^{2}+\lambda _{4^{\ast }}\Vert
u\Vert _{2}^{2}-\Vert u\Vert _{4^{\ast }}^{4^{\ast }}+o_{n}(1)=o_{n}(1),
\end{equation*}%
and
\begin{eqnarray*}
I_{\lambda _{p_{n}}}(u) &=&\frac{1}{2}\left( \Vert \Delta u\Vert
_{2}^{2}-\mu \Vert \nabla u\Vert _{2}^{2}+\lambda _{p_{n}}\Vert u\Vert
_{2}^{2}\right) -\frac{1}{p_{n}}\Vert u\Vert _{p_{n}}^{p_{n}} \\
&=&\frac{1}{2}\left( \Vert \Delta u\Vert _{2}^{2}-\mu \Vert \nabla u\Vert
_{2}^{2}+\lambda _{4^{\ast }}\Vert u\Vert _{2}^{2}\right) -\frac{1}{4^{\ast }%
}\Vert u\Vert _{4^{\ast }}^{4^{\ast }}+o_{n}(1) \\
&=&I_{\lambda_{4^{\ast}}}(u)+o_{n}(1),
\end{eqnarray*}%
which implies that $\theta _{\lambda _{p_{n}}}\leq I_{\lambda _{4^{\ast
}}}(u)+o_{n}(1)\leq \theta _{\lambda _{4^{\ast }}}+\epsilon +o_{n}(1)$.
Hence, $\limsup_{p_{n}\rightarrow (4^{\ast })^{-}}\theta _{\lambda
_{p_{n}}}\leq \theta _{\lambda _{4^{\ast }}}$ holds by the arbitrary of $%
\epsilon $. Similarly, we also have $\theta _{\lambda _{4^{\ast }}}\leq
\liminf_{p_{n}\rightarrow (4^{\ast })^{-}}\theta _{\lambda _{p_{n}}}$. The
proof is complete.
\end{proof}

\textbf{We now give the proof of Theorem \ref{T1.5}: }Let
\begin{equation*}
\lambda \in \{\lambda (u):u\ \text{is a local minimizer of problem}\ \text{(%
\ref{E3}) in the set }\mathcal{M}_{a}\}
\end{equation*}%
and $u_{\lambda }$ be a nontrivial critical point of $I_{\lambda }$ such
that $I_{\lambda }(u_{\lambda })=\theta _{\lambda }$. Clearly, $u_{\lambda
}\in \mathcal{P}_{\Vert u_{\lambda }\Vert _{2}}$ and
\begin{equation}
\theta _{\lambda }\leq \alpha _{p}(a)+\frac{\lambda }{2}a^{2}.  \label{E72}
\end{equation}%
Now we prove that $\Vert u_{\lambda }\Vert _{2}=a$. Suppose on the contrary,
we split it into the following two cases:

Case $(i):\Vert u_{\lambda }\Vert _{2}>a$. Then there exists a positive
constant $s_{\lambda }<1$ such that $\Vert s_{\lambda }u_{\lambda }\Vert
_{2}=a$. Next we claim that $\Vert s_{\lambda }\Delta u_{\lambda }\Vert
_{2}<\rho_{0}$. By (\ref{E72}) and $\lambda >0$, we have
\begin{align*}
\alpha _{p}(a)+\frac{\lambda }{2}a^{2}& \geq \Psi _{\mu ,p}(u_{\lambda })+%
\frac{\lambda }{2}\Vert u_{\lambda }\Vert _{2}^{2} \\
& >\Psi _{\mu ,p}(u_{\lambda })+\frac{\lambda }{2}a^{2},
\end{align*}%
which implies that $\Psi_{\mu,p}(u_{\lambda })<\alpha _{p}(a)<0$. In view of
$u_{\lambda }\in \mathcal{P}_{\Vert u_{\lambda }\Vert _{2}}$, we calculate
that
\begin{align*}
0>\Psi _{\mu ,p}(u_{\lambda })=& \frac{p\gamma _{p}-2}{2p\gamma _{p}}\Vert
\Delta u_{\lambda }\Vert _{2}^{2}-\frac{\mu (p\gamma _{p}-1)}{2p\gamma _{p}}%
\Vert \nabla u_{\lambda }\Vert _{2}^{2} \\
\geq & \frac{p\gamma _{p}-2}{2p\gamma _{p}}\Vert \Delta u_{\lambda }\Vert
_{2}^{2}-\frac{\mu a(p\gamma _{p}-1)}{2p\gamma _{p}}\Vert \Delta u_{\lambda
}\Vert _{2},
\end{align*}%
leading to $\Vert \Delta u_{\lambda }\Vert _{2}<\frac{(p\gamma _{p}-1)\mu }{%
2(p\gamma _{p}-2)}\Vert u_{\lambda }\Vert _{2}$. This shows that $\Vert
s_{\lambda }\Delta u_{\lambda }\Vert _{2}<\frac{(p\gamma _{p}-1)\mu a}{%
p\gamma _{p}-2}<\rho_{0}$. So we have $s_{\lambda }u_{\lambda }\in
\mathcal{M}_{a}$ and
\begin{equation}
\Psi _{\mu ,p}(s_{\lambda }u_{\lambda })\geq \alpha _{p}(a).  \label{E73}
\end{equation}%
Set
\begin{equation*}
\bar{g}(s):=I_{\lambda }(su_{\lambda })=\frac{s^{2}}{2}\Vert \Delta
u_{\lambda }\Vert _{2}^{2}+\frac{\lambda s^{2}}{2}\Vert u_{\lambda }\Vert
_{2}^{2}-\frac{\mu s^{2}}{2}\Vert \nabla u_{\lambda }\Vert _{2}^{2}-\frac{%
s^{p}}{p}\Vert u_{\lambda }\Vert _{p}^{p}\text{ for }s>0.
\end{equation*}%
Since $\bar{g}^{\prime }(1)=\langle I_{\lambda }^{\prime }(u_{\lambda
}),u_{\lambda }\rangle =0$, it is easy to obtain that $\bar{g}(s)$ is
increasing on $(0,1)$ and is decreasing on $(1,\infty )$. Moreover, $\bar{g}%
(1)=\max\limits_{s\in (0,\infty )}\bar{g}(s)$. Then according to (\ref{E72}%
) and (\ref{E73}), there holds
\begin{align*}
\alpha _{p}(a)+\frac{\lambda a^{2}}{2}\geq I_{\lambda }(u_{\lambda })>&
I_{\lambda }(s_{\lambda }u_{\lambda }) \\
=& \Psi _{\mu ,p}(s_{\lambda }u_{\lambda })+\frac{\lambda a^{2}}{2} \\
\geq & \alpha _{p}(a)+\frac{\lambda a^{2}}{2}.
\end{align*}%
This is impossible.

Case $(ii):\Vert u_{\lambda }\Vert _{2}<a$. Then there exists a positive
constant $s_{\lambda }>1$ such that $\Vert s_{\lambda }u_{\lambda }\Vert
_{2}=a$. If $\Psi _{\mu ,p}(u_{\lambda })\leq 0$, similar to the argument of
Case $(i)$, we easily obtain a contradiction. Otherwise, we may assume that $%
\Psi _{\mu ,p}(u_{\lambda })>0$. By Lemma \ref{T4.1} $(ii)$ and (\ref{E72}%
), we have
\begin{equation}
s_{\lambda }^{-(2+4b)}M_{\mu ,p}(a)+\frac{\lambda s_{\lambda }^{-2}}{2}%
a^{2}\leq M_{\mu ,p}\left( s_{\lambda }^{-1}a\right) +\frac{\lambda
s_{\lambda }^{-2}}{2}a^{2}<\alpha _{p}(a)+\frac{\lambda }{2}a^{2}.
\label{E74}
\end{equation}%
A direct calculation shows that
\begin{eqnarray*}
\theta _{\lambda } &=&\frac{1}{2}\left( \Vert \Delta u_{\lambda }\Vert
_{2}^{2}-\mu \Vert \nabla u_{\lambda }\Vert _{2}^{2}+\lambda \Vert
u_{\lambda }\Vert _{2}^{2}\right) -\frac{1}{p}\Vert u_{\lambda }\Vert
_{p}^{p} \\
&&-\frac{1}{2}\left( \Vert \Delta u_{\lambda }\Vert _{2}^{2}-\mu \Vert
\nabla u_{\lambda }\Vert _{2}^{2}+\lambda \Vert u_{\lambda }\Vert
_{2}^{2}\right) -\frac{1}{2}\Vert u_{\lambda }\Vert _{p}^{p} \\
&=&\frac{p-2}{2p}\Vert u_{\lambda }\Vert _{p}^{p}>0,
\end{eqnarray*}%
which implies that $\frac{\lambda }{2}a^{2}+\alpha _{p}(a)>\theta _{\lambda
}>0$ by (\ref{E72}). Let%
\begin{eqnarray}
\overline{L}(s) &=&s^{-(2+4b)}M_{\mu ,p}(a)+\frac{\lambda }{2}a^{2}\left(
s^{-2}-1\right) -\alpha _{p}(a)  \notag \\
&=&s^{-2}\left[ s^{-4b}M_{\mu ,p}(a)+\frac{\lambda }{2}a^{2}-s^{2}\left(
\frac{\lambda }{2}a^{2}+\alpha _{p}(a)\right) \right] \text{ for }s\geq 1.
\label{E75}
\end{eqnarray}%
Clearly, $\overline{L}(1)>0.$ Since $-4b>2$ for $\bar{p}<p<4^{\ast }$ and $%
M_{\mu ,p}(a)-\frac{\lambda }{2}a^{2}-\alpha _{p}(a)>0$, it follows from (%
\ref{E75}) that $\overline{L}(s)>0$ for any $s\geq 1$. This contradicts
with (\ref{E74}).

Now we consider the case of $p=4^{\ast }$. Note that $-4b=2$ for $p=4^{\ast
}.$ Then by (\ref{E75}), if
\begin{equation*}
M_{\mu ,4^{\ast }}(a)-\frac{\lambda }{2}a^{2}-\alpha _{4^{\ast }}(a)>0,
\end{equation*}%
there holds $\overline{L}(s)>0$ for any $s>1.$ This contradicts with (\ref%
{E74}). Next, we claim that
\begin{equation*}
M_{\mu ,4^{\ast }}(a)-\frac{\lambda }{2}a^{2}-\alpha _{4^{\ast }}(a)>0.
\end{equation*}%
In fact, it follows from Lemmas \ref{T2.7}, \ref{T4.1} and Theorem \ref{T1.4}
that
\begin{equation}
\Psi _{\mu ,p}(u_{\lambda })\geq M_{\mu ,4^{\ast }}\left( \Vert u_{\lambda
}\Vert _{2}\right) \geq M_{\mu ,4^{\ast }}\left( a\right) \ \text{and }%
\lim\limits_{\mu \rightarrow 0^{+}}M_{\mu ,4^{\ast }}\left( a\right)
=M_{4^{\ast}}(a)>0.  \label{E76}
\end{equation}%
Moreover, by Theorem \ref{T1.2} one has
\begin{equation}
\lim\limits_{\mu \rightarrow 0^{+}}\alpha _{4^{\ast }}(a)=\lim\limits_{\mu
\rightarrow 0^{+}}\lambda =0.  \label{E77}
\end{equation}%
Let $L(\mu )=M_{\mu ,4^{\ast }}(a)-\alpha _{4^{\ast }}(a)-\frac{\lambda }{2}%
a^{2}$. Then by (\ref{E76}) and (\ref{E77}), there exists $\mu
_{a}^{\ast }\in \left( 0,\left( C_{0}a^{-\frac{8}{N-4}}\right) ^{\frac{N-4}{8%
}}\right] $ such that for any $\mu \in (0,\mu _{a}^{\ast })$, there holds $%
L(\mu )>0$. Therefore, we have $\Vert u_{\lambda }\Vert _{2}=a$ and $%
u_{\lambda }\in \mathcal{M}_{a}$.

$(ii)$ From $(i)$, we conclude that%
\begin{equation}
\alpha _{p}(a)+\frac{\lambda _{p}}{2}a^{2}=\theta _{\lambda }  \label{E78}
\end{equation}%
for any $\bar{p}<p<\min \{4,4^{\ast }\},$ where $\lambda _{p}$ is the
related Lagrange multiplier.

Now we consider the case of $p=4^{\ast }$. Let
\begin{equation*}
\lambda _{r}\in \{\lambda (u):u\ \text{is a local minimizer of problem}\
\text{(\ref{E3})}\ \text{with }\bar{p}<p<4^{\ast }\text{ in the set}\
\mathcal{M}_{a}\}
\end{equation*}%
and $r\rightarrow (4^{\ast })^{-}$. Clearly, $\lambda _{r}$ is bounded. Then
there exists $\lambda _{4^{\ast }}\in \mathbb{R}$ such that up to a
subsequence, $\lim\limits_{r\rightarrow (4^{\ast })^{-}}\lambda _{r}=\lambda
_{4^{\ast }}$. By (\ref%
{E72}), (\ref{E78}) and Lemmas \ref{T5.1}--\ref{T5.2}, we have
\begin{equation}
\alpha _{4^{\ast }}(a)+\frac{\lambda _{4^{\ast }}}{2}a^{2}=\lim\limits_{r%
\rightarrow (4^{\ast })^{-}}\left( \alpha _{r}(a)+\frac{\lambda _{r}}{2}%
a^{2}\right) =\lim\limits_{r\rightarrow (4^{\ast })^{-}}\theta _{\lambda
_{r}}=\theta _{\lambda _{4^{\ast }}}.  \label{E79}
\end{equation}%
Next, we claim that
\begin{equation*}
\lambda _{4^{\ast }}\in \{\lambda (u):u\ \text{is a local minimizer of
problem}\ \text{(\ref{E3})}\ \text{with}\ p=4^{\ast }\ \text{in the set}\
\mathcal{M}_{a}\}.
\end{equation*}%
Define $u_{r}\in \mathcal{M}_{a}$ such that $\Psi _{\mu ,r}(u_{r})=\alpha
_{r}(a)$. Clearly, $\Vert u_{r}\Vert _{H^{2}}$ is bounded. Similar to the
proof of Proposition \ref{T3.4}, there exists $u\in H^{2}(\mathbb{R}%
^{N})\backslash \{0\}$ such that up to a subsequence, $u_{r}\rightarrow u$
in $(\mathbb{R}^{N})$ as $r\rightarrow (4^{\ast })^{-}$. Moreover, there
holds $\Psi _{\mu ,4^{\ast }}(u)=\alpha _{4^{\ast }}(a)$. A direct
calculation shows that
\begin{equation*}
0=\lim\limits_{r\rightarrow (4^{\ast })^{-}}\left( \Vert \Delta u_{r}\Vert
_{2}^{2}-\mu \Vert \nabla u_{r}\Vert _{2}^{2}+\lambda _{r}\Vert u_{r}\Vert
_{2}^{2}-\Vert u_{r}\Vert _{r}^{r}\right) =\Vert \Delta u\Vert _{2}^{2}-\mu
\Vert \nabla u\Vert _{2}^{2}+\lambda _{4^{\ast }}\Vert u\Vert _{2}^{2}-\Vert
u\Vert _{4^{\ast }}^{4^{\ast }},
\end{equation*}%
which implies that $u$ is a local minimizer of problem (\ref{E3}) with $%
p=4^{\ast }$ in the set $\mathcal{M}_{a}$. Combining with (\ref{E79}), we
complete the proof of Theorem \ref{T1.5} $(ii)$.

\section{The dynamical behavior}

First of all, we use Corollary \ref{T3.5} and the standard
concentration-compactness argument mentioned in \cite{C1982,C2003} to prove
Theorem \ref{T1.7}. We give some notations and some results in \cite{P2007}%
. Assume that $2\leq q,r\leq \infty $, $(q,r,N)\neq (2,\infty ,4)$, and
\begin{equation*}
\frac{4}{q}+\frac{N}{r}=\frac{N}{2}.
\end{equation*}%
We say that a pair $(q,r)$ is Schr\"{o}dinger admissible. Let $L^{q}=L^{q}(%
\mathbb{R}^{N})$ be the Lebesgue space, and let $L^{r}(I,L^{q}):I\subset
\mathbb{R}\mapsto L^{q}$ be the space of measurable functions whose norm is
finite, where
\begin{equation*}
\Vert u\Vert _{L^{q}(I,L^{r})}=\left( \int_{I}\Vert u(t)\Vert
_{L^{r}}^{q}dt\right) ^{1/q}.
\end{equation*}%
Set
\begin{equation*}
u(t):=e^{it(\Delta ^{2}+\mu \Delta )}\psi _{0}+i\int_{0}^{t}e^{i(t-s)(\Delta
^{2}+\mu \Delta )}|u(s)|^{p-2}u(s)ds.  \label{e4.1}
\end{equation*}%
Then $u\in C(I,H^{2})$ is a solution of the Cauchy problem (\ref{E1}).

\begin{proposition}
\label{T6.1} (\cite[Proposition 3.1]{P2007}) Let $u\in C(I,H^{-4})$ be a
solution of
\begin{equation*}
\begin{cases}
i\partial _{t}u+\Delta ^{2}u=h, \\
u(0)=\psi _{0}\in H^{2}(\mathbb{R}^{N}).%
\end{cases}%
\end{equation*}%
Then for any $S$-admissible pair $(q,r)$ and $(\overline{q},\overline{r})$,
there holds%
\begin{equation}
\Vert u\Vert _{L^{q}(I,L^{r})}\leq C\left( \Vert \psi _{0}\Vert _{2}+\Vert
h\Vert _{L^{\overline{q}}(I,L^{\overline{r}})}\right) ,  \label{E80}
\end{equation}%
whenever the right hand side in (\ref{E80}) is finite.
\end{proposition}

\begin{proposition}
\label{T6.2} $(i)$(\cite[Propositions 4.1]{P2007}) Let $\bar{p}<p<4^{\ast }$%
. For any $\psi _{0}\in H^{2}(\mathbb{R}^{N})$, there exists $T>0$ and a
unqiue solution $u\in C([0,T);H^{2}(\mathbb{R}^{N}))$ to the Cauchy problem (%
\ref{E1}) with initial datum $\psi _{0}$ such that the mass and the energy
are conserved along time, that is, for any $t\in \lbrack 0,T)$,
\begin{equation*}
\Psi _{\mu ,p}(u)=\Psi _{\mu ,p}(\psi _{0})\ \text{and }\Vert u\Vert
_{2}=\Vert \psi _{0}\Vert _{2}.
\end{equation*}%
Moreover, either $T=\infty $ or $\lim_{t\rightarrow T^{-}}\Vert \Delta
u\Vert _{2}=\infty $.\newline
$(ii)$ (\cite[Propositions 5.1]{P2007}) Let $p=4^{\ast }$. Then there exists
$C_0>0$ such that for any initial data $\psi _{0}\in H^{2}(\mathbb{R}%
^{N})$, and any interval $I_{T}=[0,T]$ with $T\leq 1$, if
\begin{equation*}
\Vert e^{it(\Delta ^{2}+\mu \Delta )}\psi _{0}\Vert _{Z_{T}}<C_0,
\end{equation*}%
where $Z_{T}$ is as in $\Vert \cdot \Vert _{Z_{T}}:=\Vert \cdot \Vert _{L^{%
\frac{2(N+4)}{N-4}}\left( I_{T},L^{\frac{2N(N+4)}{N^{2}-2N+8}}\right) }$,
then there exists a unique solution $u\in C(I,H^{2}(\mathbb{R}^{N}))$ of the
Cauchy problem (\ref{E1}) with initial date $\psi _{0}$. This solution has
conserved mass and energy:
\begin{equation*}
\Psi _{\mu ,4^{\ast}}(u)=\Psi _{\mu ,4^{\ast}}(\psi _{0})\text{ and}\ \Vert u\Vert
_{2}=\Vert \psi _{0}\Vert _{2}\ \text{for all}\ t\in I_{T}.
\end{equation*}
\end{proposition}

Using the Strichartz type estimate, we show that the set $\mathbf{B}_{\rho_{0}}$ is orbitally stable.

\begin{theorem}
\label{T6.3} Let $v\in \mathbf{B}_{\rho_{0}}$. Then for each $\epsilon >0$%
, there is $\kappa >0$ such that for any $\psi \in H^{2}(\mathbb{R}^{N})$
satisfying $\Vert \psi -v\Vert _{H^{2}}<\kappa $, there holds
\begin{equation}
\sup\limits_{t\in \lbrack 0,T_{\psi })}dist_{H^{2}}(\zeta _{\psi }(t),%
\mathbf{B}_{\rho_{0}})<\epsilon .  \label{E81}
\end{equation}%
Moreover, for every $t\in \lbrack 0,T_{\psi })$, the function $\zeta _{\psi
}(t)$ satisfies
\begin{equation}
\zeta _{\psi }(t)=\xi _{a}(t)+\nu (t),  \label{E82}
\end{equation}%
where $\xi _{a}(t)\in \mathbf{B}_{\rho_{0}}\ $and $\Vert \nu (t)\Vert
_{H^{2}}<\epsilon $.
\end{theorem}

\begin{proof}
Similar to the argument of \cite[Theorem 4.1]{J2022}, we assume that there
exist sequences $\{\kappa _{n}\}\subset \mathbb{R}^{+}$ satisfying $\kappa
_{n}=o_{n}(1)$ and $\{\psi _{n}\}\subset H^{2}(\mathbb{R}^{N})$ satisfying
\begin{equation*}
\Vert \psi _{n}-v\Vert _{H^{2}}<\kappa _{n}\ \text{and}\ \sup\limits_{t\in
\lbrack 0,T_{\psi _{n}})}dist_{H^{2}}(\zeta _{\psi _{n}}(t),\mathbf{B}%
_{\rho_{0}})>\epsilon _{0},
\end{equation*}%
for some $\epsilon_0 >0$. Since the function $\Psi _{\mu ,p}$ is continuous,
it follows from $\Vert \psi _{n}\Vert _{2}^{2}\rightarrow a$ that $\Psi
_{\mu ,p}(\psi _{n})\rightarrow \alpha _{p}(a)$ as $n\rightarrow \infty $.
Then according to the conservation law, we have $\zeta _{\psi _{n}}\in
\mathcal{M}_{a}$ for any $t\in \lbrack 0,T_{\psi _{n}})$ and $n\in \mathbb{N}
$ large enough. Otherwise, there is $t_{0}>0$ such that $\Vert \Delta \zeta
_{\psi _{n}}(t_{0})\Vert _{2}=\rho_{0}$ and thus $\Psi _{\mu ,p}(\zeta
_{\psi _{n}})>0>\alpha _{p}(a)$ by Lemma \ref{T2.2} $(i)$. This is
impossible.

Next, we denote $\zeta _{n}:=\zeta _{\psi _{n}}(t_{n}),$ where $t_{n}>0$ is
the first time at which $dist_{H^{2}}(\zeta _{\psi _{n}}(t_{n}),\mathbf{B}%
_{\rho_{0}})=\epsilon _{0}$ is satisfied. Since the sequence $\{\zeta
_{n}\}\subset \mathcal{M}_{a}$ satisfies $\Vert \zeta _{n}\Vert
_{2}^{2}=a^{2}$ and $\Psi _{\mu ,p}(\zeta _{n})\rightarrow \alpha _{p}(a)$,
we know that $\zeta _{n}$ converges to an element of $\mathbf{B}_{\rho_{0}}$ after translation by using Corollary \ref{T3.5}, which contradicts
with $dist_{H^{2}}(\zeta _{n},\mathbf{B}_{\rho_{0}})=\epsilon _{0}>0$.
The proof is complete.
\end{proof}

Now, we prove $T_{\psi}=\infty$, which leads directly to the conclusion of
Theorem \ref{T1.7}.
\begin{proposition}
\label{T6.4} Assume that $(q,r)$ is an admissible pair
and $X\subset H^{2}(\mathbb{R}^{N})\setminus \{0\}$ is
a compact set after translation. Then for each $\sigma >0$, there exist $%
\epsilon =\epsilon (\sigma )>0$ and $T=T(\sigma )>0$ such that
\begin{equation*}
\sup\limits_{\{\psi \in H^{2}:dist_{H^{2}}(\psi ,X)<\epsilon \}}\Vert
e^{it(\Delta ^{2}+\mu \Delta )}\psi \Vert _{L^{q}(I_T,L^{r})}<\sigma .
\end{equation*}
\end{proposition}

\begin{proof}
We first show that for every $\sigma >0$, there is a $T>0$ such that
\begin{equation}
\sup\limits_{\psi \in X}\Vert e^{it(\Delta ^{2}+\mu \Delta )}\psi \Vert
_{L^{q}(I_T,L^{r})}<\frac{\sigma }{2}.  \label{E83}
\end{equation}%
Otherwise, for $\sigma_{0}>0$, there exist sequences $\{\psi _{n}\}\subset X$
and $\{T_{n}\}\subset \mathbb{R}^{+}$ such that $T_{n}\rightarrow 0$ and
\begin{equation}
\Vert e^{it(\Delta ^{2}+\mu \Delta )}\psi _{n}\Vert _{L^{q}(I_{T_n},L^{r})}\geq \sigma_{0}.  \label{E84}
\end{equation}%
In view of compactness of the set $X$, up to a subsequence, there exist a
sequence $\{x_{n}\}\subset \mathbb{R}^{N}$ and $\psi \in H^{2}(\mathbb{R}%
^{N})$ such that
\begin{equation*}
\overline{\psi }_{n}(\cdot ):=\psi _{n}(\cdot -x_{n})\rightarrow \psi (\cdot
)\ \text{in}\ H^{2}(\mathbb{R}^{N}).
\end{equation*}%
Then using (\ref{E84}) and the translation invariance of Strichartz
estimate, we have
\begin{equation}
\Vert e^{it(\Delta ^{2}+\mu \Delta )}\overline{\psi }_{n}\Vert _{L^{q}(I_{\overline{T}},L^{r})}
\rightarrow \Vert e^{it(\Delta ^{2}+\mu \Delta )}\psi \Vert
_{L^{q}(I_{\overline{T}},L^{r})}\ \text{for any}\ \overline{T}>0  \label{E85}
\end{equation}%
and
\begin{equation}
\Vert e^{it(\Delta ^{2}+\mu \Delta )}\overline{\psi }_{n}\Vert
_{L^{q}(I_{T_n},L^{r})}=\Vert e^{it(\Delta ^{2}+\mu \Delta )}\psi _{n}\Vert
_{L^{q}(I_{T_n},L^{r})}\geq \sigma _{0}.  \label{E86}
\end{equation}%
By (\ref{E80}), one has $e^{it(\Delta ^{2}+\mu \Delta )}\psi \in L^{q}(I_{1},L^{r})$.
Then it follows from the Dominated Convergence Theorem that $\Vert
e^{it(\Delta ^{2}+\mu \Delta )}\psi \Vert _{L^{q}(I_{\overline{T}},L^{r})}\rightarrow 0$
as $\overline{T}\rightarrow 0$. Thus, we can choose a $\overline{T}>0$ such
that
\begin{equation*}
\Vert e^{it(\Delta ^{2}+\mu \Delta )}\psi \Vert _{L^{q}(I_{\overline{T}},L^{r})}<\sigma
_{0}.
\end{equation*}%
Hence, by (\ref{E84})-(\ref{E86}), we reach a contradiction.

Now, fix $T>0$ such that (\ref{E83}) holds. For any $h\in H^{2}(\mathbb{R}%
^{N})$, it follows from (\ref{E80}) that $\Vert e^{it(\Delta ^{2}+\mu
\Delta )}h\Vert _{L^{q}(I_T,L^{r})}\leq C\Vert h\Vert_{L^{\overline{q}}(I_T,L^{\overline{r}})}$. Let $\Vert h\Vert_{L^{\overline{q}}(I_T,L^{\overline{r}})}<\frac{\sigma }{2C}:=\epsilon $, we have
\begin{equation}
\Vert e^{it(\Delta ^{2}+\mu \Delta )}h\Vert _{L^{q}(I_T,L^{r})}<\frac{\sigma }{2}.
\label{E87}
\end{equation}%
Then by (\ref{E83}) and (\ref{E87}), we conclude that for any $\psi \in X$
and $h\in H^{2}(\mathbb{R}^{N})$ satisfying $\Vert h\Vert_{L^{q}(I_T,L^{r})}<\epsilon $%
, there holds
\begin{equation*}
\Vert e^{it(\Delta ^{2}+\mu \Delta )}\psi \Vert _{L^{q}(I_T,L^{r})}+\Vert e^{it(\Delta
^{2}+\mu \Delta )}h\Vert _{L^{q}(I_T,L^{r})}<\sigma .
\end{equation*}%
The proof is complete.
\end{proof}

\begin{theorem}
\label{T6.5} Assume that $X\subset H^{2}(\mathbb{R}^{N})\setminus \{0\}$ is
a compact set up to translation. Then there exist $\epsilon _{0},T_{0}>0$
such that the Cauchy problem (\ref{E1}) has a unique solution $\psi $ on
the time interval $[0,T_{0}]$, and $\psi $ satisfies $dist_{H^{2}}(\psi
,X)<\epsilon _{0}$.
\end{theorem}

\begin{proof}
By Propositions \ref{T6.2} and \ref{T6.4}, we easily obtain the conclusion.
\end{proof}

\begin{theorem}
\label{T6.6} Let $X=\mathbf{B}_{\rho_{0}}\subset H^{2}(\mathbb{R}^{N})$.
Then there exists $\epsilon _{0}>0$ such that the Cauchy problem (\ref{E1}%
), where $dist_{H^{2}}(\psi ,\mathbf{B}_{\rho_{0}})<\epsilon _{0}$, has a
unique solution on the time interval $[0,\infty )$.
\end{theorem}

\begin{proof}
By Corollary \ref{T3.5}, we know that $\mathbf{B}_{\rho_{0}}$ is compact,
up to translation. Thus, it follows from Theorems \ref{T6.3} and \ref{T6.5}
that we can choose $\kappa =\kappa _{0}>0$ and $\epsilon =\epsilon _{0}>0$ such
that (\ref{E81}) and (\ref{E82}) hold. According to Theorem \ref{T6.3}, we
derive the solution $\zeta _{\psi }(t)$ where $dist_{H^{2}}(\psi ,\mathbf{B}_{\rho_{0}})<\kappa _{0}$ satisfies $dist_{H^{2}}(\zeta _{\psi }(t),\mathbf{B}_{\rho_{0}})<\epsilon _{0}$, until the maximum time of existence $%
T_{\psi }\geq T_{0}$. At any time in $(0,T_{\psi })$, we can continue to use
Theorem \ref{T6.5} to ensure a consistent additional existence time $T_{0}>0$%
. Thererfore, we have $T_{\psi }=\infty $. The proof is complete.
\end{proof}

\textbf{Proof of Theorem \ref{T1.7}:} According to Theorems \ref{T6.3} and %
\ref{T6.6}, we directly conclude that $\mathbf{B}_{\rho_{0}}$ is
orbitally stable.

\textbf{Proof of Theorem \ref{T1.8}}: Let $u_{\mu }^{-}\in S_{a}$ be a
solution to problem (\ref{E3}) at the level $M_{\mu ,p}(a)>0$ by Theorem %
\ref{T1.4}. It follows from Lemma \ref{T2.7} that there exists $s^{\ast }>0$
such that $\Psi _{\mu ,p}((u_{\mu }^{-})_{s})<M_{\mu ,p}(a)$ and $Q_{p}((u_{\mu
}^{-})_{s})>0$ for any $s\in (s^{\ast },1)$, which implies that $\Gamma_{a}\neq \emptyset $. Hence, we conclude that
\begin{equation*}
\Vert (u_{\mu }^{-})_{s}-u_{\mu }^{-}\Vert _{H^{2}}\rightarrow 0,\text{ }%
\Psi _{\mu ,p}((u_{\mu }^{-})_{s})\rightarrow M_{\mu ,p}(a)\text{ and }%
Q_{p}((u_{\mu }^{-})_{s})\rightarrow 0^{+}\text{ as }s\rightarrow 1^{-}.
\end{equation*}%
Set $\psi _{0}:=(u_{\mu }^{-})_{s}\in \Gamma _{a}$. By Proposition \ref{T6.2}%
, there exists a unique $\psi (t)\in C\left( [0,T];H^{2}(\mathbb{R}%
^{N})\right) $ to the Cauchy problem (\ref{E1}) with the initial datum $%
\psi _{0}$. We now prove that $\psi $ exists globally in time, i.e., $%
T=\infty $. Suppose on the contrary, we assume $T<\infty $. By Proposition %
\ref{T6.2}, one has
\begin{equation}
\lim\limits_{t\rightarrow T^{-}}\int_{\mathbb{R}^{N}}|\Delta \psi
(t)|^{2}dx=\infty .  \label{E88}
\end{equation}%
Through conservation of energy and mass, we obtain that $\Psi _{\mu ,p}(\psi
)=\Psi _{\mu ,p}(\psi _{0})$ and $\Vert \psi _{0}\Vert _{2}=\Vert \psi \Vert
_{2}=a$ for any $t\in \lbrack 0,T]$. Then we calculate that
\begin{equation*}
\frac{1}{2p\gamma _{p}}Q_{p}(\psi )=\Psi _{\mu ,p}(\psi _{0})-\frac{p\gamma
_{p}-2}{2p\gamma _{p}}\Vert \Delta \psi \Vert _{2}^{2}+\frac{\mu (p\gamma
_{p}-1)}{2p\gamma _{p}}\Vert \nabla \psi \Vert _{2}^{2},
\end{equation*}%
which implies that
\begin{equation}
\frac{1}{2p\gamma _{p}}Q_{p}(\psi )\leq \Psi _{\mu ,p}(\psi _{0})-\frac{p\gamma
_{p}-2}{2p\gamma _{p}}\Vert \Delta \psi \Vert _{2}^{2}+\frac{\mu a(p\gamma
_{p}-1)}{2p\gamma _{p}}\Vert \Delta \psi \Vert _{2}.  \label{E89}
\end{equation}%
Using (\ref{E88}) and (\ref{E89}), one has $\lim\limits_{t\rightarrow
T^{-}}Q_{p}(\psi )=-\infty $. Since $Q_{p}(\psi _{0})>0$ and $\lim\limits_{t%
\rightarrow T^{-}}Q_{p}(\psi )=-\infty $, by continuity, there exists $t_{0}\in
(0,T)$ such that $Q_{p}(\psi (t_{0}))=0$, $\lim\limits_{t\rightarrow
t_{0}^{-}}Q_{p}(\psi (t))=0^{+}$ and $\lim\limits_{t\rightarrow t_{0}^{+}}Q_{p}(\psi
(t))=0^{-}$. According to Lemma \ref{T2.7}, we have $\psi (t_{0})\in
\mathcal{P}_{a}^{+}$. From the definition of $M_{\mu ,p}(a)$ it follows that
$\Psi _{\mu ,p}(\psi (t_{0}))\geq M_{\mu ,p}(a),$ which contradicts with $%
\Psi _{\mu ,p}(\psi (t_{0}))=\Psi _{\mu ,p}(\psi _{0})<M_{\mu ,p}(a)$. The
proof is complete.

\section*{Acknowledgments}

J. Sun was supported by the National Natural Science Foundation of China
(Grant No. 12371174) and Shandong Provincial Natural Science Foundation
(Grant No. ZR2020JQ01).

\end{document}